\documentclass[a4paper,reqno]{amsart}%

\usepackage[%
backend=bibtex,%
style=alphabetic,%
natbib=false,%
isbn=false, 
doi=false, 
url=false, 
eprint=true, 
giveninits=true, 
backref=false, 
maxbibnames=99,%
]{biblatex}%
\AtEveryBibitem{%
  \ifentrytype{unpublished}{}{\clearfield{note}}%
  }%
\renewbibmacro{in:}{} 
\usepackage{doi}%

\usepackage{amssymb}%
\usepackage{xparse}%
\usepackage{mathtools}%
\usepackage{tikz-cd}%
\usepackage{graphicx}%
\usepackage{eucal}%

\usepackage{hyperref}%
\hypersetup{%
  colorlinks,%
  linkcolor={red!80!black},%
  citecolor={blue!80!black},%
  urlcolor={blue!80!black}%
}%

\usepackage{thmtools}%

\usepackage{cleveref}%

\numberwithin{equation}{section}%

\declaretheorem[style=theorem,sibling=equation]{theorem}%
\declaretheorem[style=theorem,numbered=no,name=Theorem]{theorem*}%
\declaretheorem[style=theorem,sibling=equation]{lemma}%
\declaretheorem[style=theorem,sibling=equation]{proposition}%
\declaretheorem[style=theorem,sibling=equation]{corollary}%
\declaretheorem[style=definition,sibling=equation]{setup}%
\declaretheorem[style=definition,sibling=equation]{notation}%
\declaretheorem[style=definition,sibling=equation]{definition}%
\declaretheorem[style=definition,sibling=equation,name={Definition-Proposition}]{defprop}%
\declaretheorem[style=definition,numbered=no]{conventions}%
\declaretheorem[style=remark,sibling=equation]{remark}%
\declaretheorem[style=remark,sibling=equation]{example}%

\newcommand{\kk}{\mathbf{k}}%
\newcommand{\ZZ}{\mathbb{Z}}%
\newcommand{\noloc}{\rotatebox[origin=c]{180}{$\colon$}}%

\newcommand{\op}{\mathrm{op}}%
\NewDocumentCommand{\id}{O{}}{\mathrm{id}_{#1}}%

\NewDocumentCommand{\set}{mo}{\left\{#1\IfValueT{#2}{\,\middle|\,}#2\right\}}%

\newcommand{\GrothendieckUniverse}[1]{\mathbb{#1}}%
\newcommand{\UU}{\GrothendieckUniverse{U}}%
\newcommand{\VV}{\GrothendieckUniverse{V}}%
\newcommand{\WW}{\GrothendieckUniverse{W}}%

\NewDocumentCommand{\dgMod}{s O{} m}{\IfBooleanTF{#1}{\operatorname{dgMod}}{\mathbf{dgMod}}\!\left(#3\right)_{#2}}%
\NewDocumentCommand{\Mod}{s m}{\IfBooleanTF{#1}{#2\text{-}\operatorname{Mod}}{\operatorname{Mod}\!\left(#2\right)}}%
\NewDocumentCommand{\sMod}{s m}{\underline{\operatorname{Mod}}\!\left(#2\right)}%
\NewDocumentCommand{\Proj}{s m}{\operatorname{Proj}\!\left(#2\right)}%
\NewDocumentCommand{\Inj}{s m}{\operatorname{Inj}\!\left(#2\right)}%

\NewDocumentCommand{\stable}{m O{}}{\underline{#1}_{#2}}%

\NewDocumentCommand{\Hom}{O{}mm}{\operatorname{Hom}_{#1}\!\left(#2,#3\right)}%
\NewDocumentCommand{\RHom}{O{}mm}{\mathbb{R}\mathbf{Hom}_{#1}\!\left(#2,#3\right)}%
\NewDocumentCommand{\dgHom}{O{}mmO{\bullet}}{\mathbf{Hom}_{#1}^{#4}\!\left(#2,#3\right)}%
\NewDocumentCommand{\GdgHom}{O{}mmO{\bullet}}{\underline{\mathbf{Hom}}_{#1}^{#4}\!\left(#2,#3\right)}%

\NewDocumentCommand{\sHom}{O{}mm}{\underline{\operatorname{Hom}}_{#1}\!\left(#2,#3\right)}%
\NewDocumentCommand{\iHom}{O{}mm}{[#1]\!\left(#2,#3\right)}%

\NewDocumentCommand{\Lotimes}{sO{}}{\IfBooleanTF{#1}{\otimes}{\otimes^{\mathbb{L}}}_{#2}}%

\NewDocumentCommand{\dgFun}{O{\kk}mm}{\mathbf{dgFun}\!\left(#2,#3\right)}%
\NewDocumentCommand{\Fun}{O{}mm}{\operatorname{Fun}_{#1}\!\left(#2,#3\right)}%
\NewDocumentCommand{\REnd}{O{}m}{\mathbb{R}\mathbf{End}_{#1}\!\left(#2\right)}%
\NewDocumentCommand{\End}{O{}m}{\operatorname{End}_{#1}\!\left(#2\right)}%

\NewDocumentCommand{\DerCat}{d<>m}{\operatorname{D}\IfValueT{#1}{_{#1}}\!\left(#2\right)}%
\NewDocumentCommand{\dgDerCat}{d<>m}{\mathbf{D}\IfValueT{#1}{_{#1}}\!\left(#2\right)}%
\NewDocumentCommand{\dgCh}{O{}m}{\mathbf{Ch}_{#1}\!\left(#2\right)}%

\newcommand{\dgCat}{\operatorname{dgCat}}%
\newcommand{\dgcat}{\operatorname{dgcat}}%
\newcommand{\Hqe}{\operatorname{Hqe}}%
\newcommand{\Hmo}{\operatorname{Hmo}}%
\newcommand{\QQ}{\mathcal{Q}}%

\NewDocumentCommand{\Suspension}{d<>O{}m}{\Sigma\IfValueT{#1}{_{#1}}^{#2}\!\left(#3\right)}%
\NewDocumentCommand{\Loops}{d<>O{}m}{\Omega\IfValueT{#1}{_{#1}}^{#2}\!\left(#3\right)}%

\newcommand{\dgcategory}[1]{\mathcal{#1}}%
\newcommand{\A}{\dgcategory{A}}%
\newcommand{\B}{\dgcategory{B}}%
\newcommand{\C}{\dgcategory{C}}%
\newcommand{\D}{\dgcategory{D}}%
\newcommand{\G}{\dgcategory{G}}%
\newcommand{\T}{\dgcategory{T}}%

\newcommand{\category}[1]{\mathcal{#1}}%
\newcommand{\E}{\category{E}}%
\renewcommand{\S}{\category{S}}%

\newcommand{\h}{\operatorname{h}}%

\RenewDocumentCommand{\H}{O{\bullet}m}{\operatorname{H}^{#1}\!\left(#2\right)}%
\NewDocumentCommand{\Z}{O{\bullet}m}{\operatorname{Z}^{#1}\!\left(#2\right)}%
\NewDocumentCommand{\BH}{O{\bullet}m}{\operatorname{B}^{#1}\!\left(#2\right)}%

\RenewDocumentCommand{\P}{mO{\bullet}}{P_{#1}^{#2}}%
\NewDocumentCommand{\dgP}{mO{\bullet}}{\mathcal{P}_{#1}}%
\NewDocumentCommand{\dgG}{mO{\bullet}}{\mathcal{G}_{#1}}%

\newcommand{\acycl}{{\text{acycl}}}%

\NewDocumentCommand{\infDerCat}{d<>m}{\mathbb{D}\IfValueT{#1}{_{#1}}\!\left(#2\right)}%
\NewDocumentCommand{\LTP}{mO{}m}{{#1}\boxtimes_{#2}{#3}}

\begin{document}%

\title[$Q$-shaped derived categories as derived categories of dg bimodules]%
{$Q$-shaped derived categories as derived categories of differential graded
  bimodules}%

\author[G.~Jasso]{Gustavo Jasso}%

\address{%
  Mathematisches Institut, Universität zu Köln, Weyertal 86-90, 50931 Köln,
  Germany}%
\email{gjasso@math.uni-koeln.de}%
\urladdr{https://gustavo.jasso.info}%



\begin{abstract}
We prove that, under mild assumptions, the $Q$-shaped derived categories
  introduced by Holm and J{\o}rgensen are equivalent to derived categories of
  differential graded bimodules over differential graded categories. This yields new derived invariance results for $Q$-shaped derived
  categories that allow us to extend known descriptions of such categories as
  derived categories of differential graded bimodules over (possibly graded)
  algebras.
\end{abstract}

\maketitle

\setcounter{tocdepth}{1}
\tableofcontents

\section{Introduction}

$Q$-shaped derived categories are recent generalisations of derived categories
of algebras introduced by Holm and J{\o}rgensen in~\cite{HJ22} based on ideas of
Iyama and Minamoto~\cite{IM15}. In this generalisation, cochain complexes of
modules over an algebra $A$ are replaced by `$Q$-shaped' analogues, where $Q$ is
a small category satisfying suitable assumptions, see~\cite{HJ24b} for a recent
survey and~\cite{DSS17,EEG08,GH10,HJ19a} for previous related work. The
$Q$-shaped derived category $\DerCat<Q>{A}$ is an algebraic triangulated
category with small coproducts and, moreover, it is known to be compactly
generated~\cite{HJ24}. Therefore, by Keller's Recognition Theorem~\cite{Kel94},
the $Q$-shaped derived category $\DerCat<Q>{A}$ is equivalent as a triangulated
category to the usual derived category $\DerCat{\A}$ of a small differential
graded category $\A$. On the other hand, for certain choices of $Q$, it is known
that there is an equivalence of triangulated categories of the form
\[
  \DerCat<Q>{A}\simeq\DerCat{A\otimes\Gamma}
\]
where $\Gamma$ is the (possibly graded) endomorphism algebra of a well-chosen
generator (for example a tilting object or an $m$-periodic tilting object) of
the $Q$-shaped derived category of the ground commutative
ring~\cite{GHJS24,IKM17,Sai23}. One reason why the latter descriptions are not
immediately apparent is that the $Q$-shaped derived category is defined as a
localisation of the category of $(A\otimes Q)$-modules, rather than as a
localisation of the category of cochain complexes thereof. The main result in
this article is that, under mild assumptions on $A$, such a description is
available for arbitrary $Q$ if one considers bimodules over differential graded
categories.

\begin{theorem}
  \label{thm:Q-shaped-bimodules} Let $\kk$ be a hereditary commutative
  noetherian ring (e.g.~$\kk=\ZZ$ or $\kk$ is a field) and $Q$ a small
  $\kk$-category satisfying the assumptions in~\Cref{setup:Q-shaped-dercats}.
  Let $A$ be a $\kk$-algebra whose underlying $\kk$-module is finitely generated
  or flat (the latter condition is always satisfied when $\kk$ is a field).
  Then, there exists a differential graded category $\dgP{S_Q(\kk)}$ --- that
  depends only on $\kk$ and $Q$ --- and an equivalence of (dg enhanced)
  triangulated categories
  \[
    \DerCat<Q>{A}\simeq\DerCat{A\otimes\dgP{S_Q(\kk)}}
  \]
  between the $Q$-shaped derived category of $A$ and the derived category of
  differential graded $(A\otimes\dgP{S_Q(\kk)})$-modules.
\end{theorem}

The strategy to prove \Cref{thm:Q-shaped-bimodules} is straightforward: The
$Q$-shaped derived category $\DerCat<Q>{A}$ is compactly generated with a
distinguished set of compact generators $S_Q(A)$~\cite[Thm.~D]{HJ24}. Moreover,
by construction $\DerCat<Q>{A}$ admits a canonical Frobenius model, hence the
subcategory $S_Q(A)\subseteq\DerCat<Q>{A}$ can be enhanced to a differential
graded category $\dgP{S_Q(A)}$ whose derived category is equivalent to
$\DerCat<Q>{A}$ as a (dg enhanced) triangulated category~\cite[Sec.~4.3]{Kel94}.
Our proof of \Cref{thm:Q-shaped-bimodules} has formal similarities with the
proof of~\cite[Thm.~A]{GHJS24} (see \Cref{lemma:psi,lemma:rho} and
\Cref{rmk:proof-of-main-thm}); it amounts to constructing a dg functor
(\Cref{prop:the-dg-functor})
\[
  A\otimes\dgP{S_Q(\kk)}\longrightarrow\dgP{S_Q(A)}
\]
that is an isomorphism when the underlying $\kk$-module of $A$ is finitely
generated and a quasi-isomorphism when it is flat. An important component behind
the proof is the existence of well-behaved complete projective resolutions of
the members of the distinguished set of compact generators
$S_Q(\kk)\subseteq\DerCat<Q>{\kk}$~\cite[Sec.~5]{HJ24}.

After recalling the necessary preliminaries in \Cref{sec:preliminaries}, the
proof of \Cref{thm:Q-shaped-bimodules} is given in \Cref{sec:main-result}, where
we also explain some of its consequences. Notably, under the flatness assumption
(\Cref{setup:Q-shaped-dercats-flat}), we prove an invariance result
(\Cref{thm:Q-shaped-derived-invariance}) that has several interesting
consequences, some of which we now highlight. It is also worth mentioning that,
in establishing all of these consequences of \Cref{thm:Q-shaped-bimodules}, we
leverage the functoriality properties of To{\"e}n's internal
$\operatorname{Hom}$ functor for the homotopy category of dg categories, as well
as the validity of the derived Eilenberg--Watts Theorem therein~\cite{Toe07}
(these are also recalled in \Cref{sec:preliminaries}). Our main result also has
an interpretation in terms of Lurie's tensor product of $\kk$-linear presentable
stable $\infty$-categories, see~\Cref{rmk:Q-shaped-LuriesTP}.

The first consequence is that the formation of $Q$-shaped derived categories
(for fixed $Q$) preserves derived equivalences of algebras.

\begin{theorem*}[{\Cref{coro:Q-shaped-derived-invariance-fixed-Q}}]
  Let $Q$ be a small $\kk$-category and $A_1$ and $A_2$ two $\kk$-algebras, and
  assume that these satisfy the assumptions in
  \Cref{setup:Q-shaped-dercats-flat}. Suppose that $A_1$ and $A_2$ are derived
  equivalent (in the usual sense). Then, there is an equivalence of (dg
  enhanced) triangulated categories
  \[
    \DerCat<Q>{A_1}\simeq\DerCat<Q>{A_2}.
  \]
\end{theorem*}

The second consequence that we wish to highlight provides a means to obtain
equivalences between $Q$-shaped derived categories for a fixed algebra $A$ and
different choices of $Q$ (see~\Cref{ex:repetitive} for instances of this
phenomenon stemming from Happel's Theorem).

\begin{theorem*}[{\Cref{coro:Q-shaped-derived-invariance-fixed-A}}]
  Let $Q_1$ and $Q_2$ be small $\kk$-categories and $A$ a $\kk$-algebra, and
  assume that these satisfy the assumptions in
  \Cref{setup:Q-shaped-dercats-flat}. Suppose that there is a equivalence of dg
  enhanced triangulated categories
  \[
    \DerCat<Q_1>{\kk}\simeq\DerCat<Q_2>{\kk}.
  \]
  Then, there is an equivalence of (dg enhanced) triangulated categories
  \[
    \DerCat<Q_1>{A}\simeq\DerCat<Q_2>{A}.
  \]
\end{theorem*}

The latter theorem can be thought of as a `$Q$-shaped analogue' of the universal
derived equivalences for posets and small categories investigated
in~\cite{DJW21,GS16,GS16a,GS18a,GS18b,Lad08,Lad08a} among others, see also
\Cref{rmk:universal-derived-equivalences}.

We also show that different choices of compact generators of $\DerCat<Q>{\kk}$
yield alternative compact generators of $\DerCat<Q>{A}$
(\Cref{thm:independence-of-generators}); in particular, silting or tilting
subcategories of $\DerCat<Q>{\kk}$ induce generating subcategories of
$\DerCat<Q>{A}$ of the same kind
(\Cref{thm:silting-generators,thm:tilting-generators}). These results extend
known descriptions of $Q$-shaped derived categories, see
\Cref{ex:m-periodic,ex:IKM,ex:GHJS} for details.

\begin{conventions}
  Throughout the article we work over a fixed commutative ring $\kk$. We fix
  Grothendieck universes $\kk\in\UU\in\VV\in\WW\in\cdots$ that satisfy the
  infinite axiom, and whose members we refer to as the \emph{small},
  \emph{large} and \emph{very large} sets, \ldots\
  respectively.\footnote{Set-theoretic issues arise since we consider large
    derived dg categories of small dg categories, and in particular since we
    need to consider To{\"e}n's internal $\operatorname{Hom}$ functor between
    these.} Unless noted otherwise, all categories we consider are
  $\kk$-categories (=categories enriched in $\kk$-modules) and all functors are
  $\kk$-linear. In all contexts, undecorated tensor products are understood to
  be taken over $\kk$, that is $\otimes=\otimes_\kk$. All modules are right
  modules.
\end{conventions}

\section{Preliminaries}
\label{sec:preliminaries}

\subsection{Differential graded categories and their derived categories}

We use the theory of differential graded categories and their derived categories
throughout the article; our main references are \cite{Kel94,Kel06}.

\subsubsection{Differential graded $\kk$-modules}
\label{subsec:dg-k-modules}

We denote the category of \emph{differential graded (=dg) $\kk$-modules} (equivalently,
cochain complexes of $\kk$-modules) by $\dgMod*{\kk}$. Thus, the objects of
$\dgMod*{\kk}$ are pairs $(V,d_V)$ consisting of a graded $\kk$-module
\[
  V=\coprod_{i\in\ZZ}V^i\qquad d_V(V^i)\subseteq V^{i+1}
\]
and a degree $1$ morphism ${d_V\colon V\to V}$ such that $d_V\circ d_V=0$. A
\emph{morphism of dg $\kk$-modules} ${f\colon (V,d_V)\to(W,d_W)}$ is a (degree 0)
morphism $f\colon V\to W$ of graded
$\kk$-modules, $f(V^i)\subseteq W^i$, that commutes with the differentials: $d_W\circ f=f\circ d_V$.

Recall that the tensor product $V\otimes W$ of two dg $\kk$-modules ${V=(V,d_V)}$
and ${W=(W,d_W)}$ is the graded $\kk$-module whose homogeneous component of degree
$i$ is the $\kk$-module
\begin{equation}
  \label{eq:tensor-product-dg-k-mods}
  (V\otimes W)^i\coloneqq\coprod_{i=j+k}V^j\otimes W^k
\end{equation}
and which is equipped with the differential
\[
  d_{V\otimes W}\coloneqq d_V\otimes\id[W]+\id[V]\otimes d_W,
\]
that is
\begin{equation}
  \label{eq:diff-tensor-product-dgkmods}
  d_{V\otimes W}(v\otimes w)=d_V(v)\otimes w+(-1)^{j}v\otimes d_W(w),\qquad v\in
  V^j.
\end{equation}

\begin{example}
  \label{ex:tensor-product-0}
  If $V=V^0$ is concentrated in degree $0$ (hence necessarily $d_V=0$) and $W$
  is arbitrary, then
  \[
    V\otimes W=\coprod_{i\in\ZZ}V\otimes W^i,\qquad d_{V\otimes W}(v\otimes w)=
    v \otimes d_W(w).
  \]
\end{example}

The tensor product of dg $\kk$-modules endows $\dgMod*{\kk}$ with a symmetric
monoidal structure with symmetry constraint given by the Koszul sign rule
\[
  V\otimes W\stackrel{\sim}{\longrightarrow}W\otimes V,\qquad v\otimes
  w\longmapsto(-1)^{ij}w\otimes v,
\]
where $v\in V^i$ and $w\in W^j$. This monoidal structure is
closed~\cite[Def.~3.3.6]{Rie14} with the following internal $\operatorname{Hom}$
functor~\cite[Sec.~2.1]{Kel06}: Given dg $\kk$-modules $V=(V,d_V)$ and
$W=(W,d_W)$, the dg $\kk$-module of morphisms between them is the graded
$\kk$-module
\[
  \dgHom[\kk]{V}{W}\coloneqq\coprod_{i\in\ZZ}\dgHom[\kk]{V}{W}[i]
\]
whose homogeneous component of degree $i$ is the $\kk$-module
\begin{equation}
  \label{eq:culprits}
  \dgHom[\kk]{V}{W}[i]\coloneqq\prod_{j\in\ZZ}\Hom[\kk]{V^j}{W^{i+j}},
\end{equation}
and which is equipped with the differential
\begin{equation}
  \label{eq:partial}
  \partial(f)\coloneqq d_W\circ f-(-1)^{i}f \circ d_V,\qquad
  f\in\dgHom[\kk]{V}{W}[i];
\end{equation}
in coordinates:
\begin{equation}
  \label{eq:partial-coords}
  \partial(f)=(d_W^{i+j}\circ f^j-(-1)^i f^{j+1}\circ d_V^j)_{j\in\ZZ}.
\end{equation}
For example, the $0$-cocycles in $\dgHom[\kk]{V}{W}$ are precisely the morphisms
of dg $\kk$-modules ${V\to W}$, and two $0$-cocycles are cohomologous if and
only if they are cochain-homotopic in the usual sense.

A \emph{differential graded category} is a category enriched in dg
$\kk$-modules~\cite[Sec.~2.2]{Kel06}, and hence all notions of enriched category
theory such as enriched functors, enriched opposite categories, \ldots, apply in
this context (see the standard reference~\cite{Kel05a} or~\cite[Ch.~3]{Rie14}
for a modern account). In slightly more detail, the morphisms between two
objects in a dg category form a dg $\kk$-module and the (graded) composition law
satisfies the (graded) Leibniz rule, corresponding to the fact that it is given
by a morphism of dg $\kk$-modules. Similarly, a dg functor $F\colon\A\to\B$
between dg $\kk$-categories is a $\kk$-linear functor that preserves the grading
and the differentials on the dg $\kk$-modules of morphisms, which is to say that
its components
\[
  F\colon\A(a_1,a_2)\longrightarrow\B(F(a_1),F(a_2)),\qquad a_1,a_2\in\A,
\]
are morphisms of dg $\kk$-modules, etc.

For example, the above internal $\operatorname{Hom}$ allows us to promote
$\dgMod*{\kk}$ to a dg category that we denote by $\dgMod{\kk}$ (compare with
\cite[p.~37]{Rie14}). Notice that the composition of homogeneous morphisms
${f\colon U\to V}$ and ${g\colon V\to W}$ with $f$ of degree $i$ is given by
\begin{equation}
  \label{eq:graded-composition}
  g\circ f=(g^{i+j}\circ f^j)_{j\in\ZZ}.
\end{equation}

\begin{remark}
  In what follows we identify dg algebras with dg categories with a single
  object in the obvious way.
\end{remark}

\begin{remark}
  The above definitions extend verbatim to define the dg category $\dgCh{\C}$ of
  cochain complexes in an additive category $\C$ (which itself need not have a
  tensor product).
\end{remark}

Each dg category $\A$ has an associated \emph{cohomology (graded) category}
$\H{\A}$ with the same objects as $\A$ and whose graded $\kk$-modules of
morphisms
\[
  \H{\A}(x,y)\coloneqq\H{\A(x,y)},\qquad x,y\in\A,
\]
are equipped with the induced composition law. The homogeneous morphisms of
{degree $0$} in $\H{\A}$ span the \emph{$0$-th cohomology category} $\H[0]{\A}$,
which is an ordinary category.\footnote{The category $\H[0]{\A}$ should not be
  confused with the underlying category $\Z[0]{\A}$ of $\A$~\cite[Def.~3.4.5]{Rie14} whose
  morphisms are rather the $0$-cocycles.}

\subsubsection{Differential graded modules over differential graded categories}
\label{subsubsec:dg-modules}

Let $\A$ be a small dg category. We denote the dg category of (right) dg
$\A$-modules by $\dgMod{\A}$. By definition, the objects of $\dgMod{\A}$ are the
dg functors\footnote{We adopt the convention that dg modules over a small dg
  category take values in the dg category $\dgMod[\UU]{\kk}$ of small dg
  $\kk$-modules, and often drop the decoration by the universe from the
  notation.}
\[
  M\colon\A^\op\to\dgMod[\UU]{\kk},\qquad a\longmapsto M_a;
\]
we denote the dg $\kk$-modules of morphisms between dg $\A$-modules by
\[
  \dgHom[\A]{M}{N}=\coprod_{i\in\ZZ}\dgHom[\A]{M}{N}[i],\qquad M,N\in\dgMod{\A},
\]
see~\cite[Digression~7.3.1]{Rie14} for an abstract construction in terms of
enriched ends (the precise formula does not play an explicit role in the
sequel). The dg Yoneda embedding
\[
  \h\colon\A\longrightarrow\dgMod{\A},\qquad a\longmapsto \h_a\coloneqq\A(-,a),
\]
is equipped with canonical isomorphisms of dg $\kk$-modules
\[
  \dgHom[\A]{\h_a}{M}\stackrel{\sim}{\longrightarrow} M_a,\qquad a\in\A,\
  M\in\dgMod{\A},
\]
that are natural in both variables~\cite[Lemma~7.3.5]{Rie14}. In particular,
there are canonical isomorphisms of dg $\kk$-modules
\[
  \dgHom[\A]{\h_a}{\h_b}\stackrel{\sim}{\longrightarrow} \A(a,b),\qquad
  a,b\in\A,
\]
so that we may (and we will) identify the dg category $\A$ with its image under
the dg Yoneda embedding~\cite[Sec.~3.1]{Kel06}.

\subsubsection{The derived category of a differential graded category}
\label{subsubsec:derived-cat}

Let $\A$ be a small dg category. A morphism of dg $\A$-modules $f\colon M\to N$
is a \emph{quasi-isomorphism} if the induced morphisms
\[
  \H{f_a}\colon\H{M_a}\longrightarrow\H{N_a},\qquad a\in\A,
\]
are isomorphisms of graded $\kk$-modules. The \emph{derived dg category of $\A$}
is the full dg subcategory $\dgDerCat{\A}$ of $\dgMod{\A}$ spanned by the
cofibrant dg modules, that is the the dg $\A$-modules $P$ with the following
property: for every object-wise surjective quasi-isomorphism of dg $\A$-modules
$L\to M$, every morphism of dg $\A$-modules $P\to M$ factors through $L$. The
\emph{derived category of $\A$} is the $0$-th cohomology category
\[
  \DerCat{\A}\coloneqq\H[0]{\dgDerCat{\A}},
\]
which is locally small but large. We recall that there are canonical
isomorphisms
\[
  \Hom[\DerCat{\A}]{\h_a}{M[i]}\stackrel{\sim}{\longrightarrow}\H[i]{M_a},\qquad
  M\in\DerCat{\A},
\]
which induce an isomorphism between the graded subcategory of the derived
category spanned by the representable dg $\A$-modules and the cohomology
category $\H{\A}$ of $\A$~\cite[Sec.~3.2]{Kel06}.

The derived category of $\A$ is a triangulated category with small coproducts
and the set
\[
  \set{\h_a\in\DerCat{\A}}[a\in\A]\subseteq\DerCat{\A}
\]
of representable dg $\A$-modules forms a set of compact
generators~\cite[Sec.~3.2--3.5]{Kel06}. The latter statement means that the
following two conditions hold:
\begin{itemize}
\item For each object $a\in\A$ the functor
  \[
    \Hom[\DerCat{\A}]{\h_a}{-}\colon\DerCat{\A}\longrightarrow\Mod{\kk}
  \]
  preserves small coproducts, which by definition means that
  $\h_a\in\DerCat{\A}$ is a \emph{compact object}.
\item A dg $\A$-module $M\in\DerCat{\A}$ is a zero object if and only if
  \[
    \forall a\in\A,\ \forall
    i\in\ZZ,\quad\Hom[\DerCat{\A}]{\h_a}{M[i]}\cong\H[i]{M_a}\stackrel{!}{=}0.
  \]
\end{itemize}
The above definitions generalise in the obvious way to define a set of compact
generators in a triangulated category with small coproducts, see for
example~\cite{Nee01}.

\begin{remark}
  If $\A=A$ is an ordinary algebra (viewed as a dg algebra concentrated in
  degree $0$), then (right) dg $A$-modules identifiy with cochain complexes of
  (right) $A$-modules and $\DerCat{A}$ is the usual derived category of $A$; the
  regular representation of $A$ (viewed as a cochain complex concentrated in
  degree $0$) is the canonical compact generator provided by the dg Yoneda
  embedding.
\end{remark}

\subsection{The homotopy theory of dg categories}

We denote by $\dgCat$ the (locally large) category of large dg categories and dg
functors between them; similarly, we denote by $\dgcat$ its (locally small)
subcategory spanned by the small dg categories.

\subsubsection{The tensor product of dg categories}

The tensor product of two dg categories $\A$ and $\B$ is the dg category
$\A\otimes\B$ with objects the pairs $(a,b)$, $a\in\A$ and $b\in\B$, and whose
dg $\kk$-modules of morphisms
\begin{equation}
  \label{eq:homs-in-tensor-product}
  (\A\otimes\B)((a_1,b_1),(a_2,b_2))\coloneqq\A(a_1,a_2)\otimes\B(b_1,b_2),
\end{equation}
are endowed with the differential $d_\A\otimes 1_\B +1_\A\otimes d_\B$, that is
\begin{equation*}
  \label{eq:diff-tensor-product}
  d(f\otimes g)\coloneqq d_\A(f)\otimes g+(-1)^{i}f\otimes d_\B(g)
\end{equation*}
for a homogeneous morphism $f\in\A(a_1,a_2)^i$ and $g\in\B(b_1,b_2)$ (compare
with~\eqref{eq:diff-tensor-product-dgkmods}). The composition law in
$\A\otimes\B$ is determined by those in $\A$ and $\B$ according the the formula
\begin{equation}
  \label{eq:comp-tensor-product}
  (f_2\otimes g_2)\circ(f_1\otimes g_1)\coloneqq (-1)^{ij}(f_2\circ f_1)\otimes(g_2\circ g_1),
\end{equation}
where $f_1$ is a homogeneous morphism of degree $i$ and $g_2$ is a homogeneous
morphism of degree $j$. The unit of the tensor product of dg categories is the
one-object dg category associated to the ground commutative ring $\kk$.

\begin{example}
  \label{ex:tensor-product-dg-cats-0}
  If $\A=A$ is an ordinary algebra and $\B$ is arbitrary then the dg
  $\kk$-modules of the tensor product $A\otimes\B$ are given by
  \[
    A\otimes\B(b_1,b_2)=\coprod_{i\in\ZZ}A\otimes\B(b_1,b_2)^i,\qquad
    b_1,b_2\in\B,
  \]
  and are endowed with the differential
  \[
    d_{A\otimes\B}(a\otimes f)=a\otimes d_{\B}(f),
  \]
  see~\Cref{ex:tensor-product-0}. Furthermore, no Koszul signs appear in the
  composition law for $A\otimes\B$ since $A$ is concentrated in degree $0$.
\end{example}

The category $\dgCat$ admits a closed symmetric monoidal structure
\[
  (\dgCat,\ \otimes,\ \dgFun{-}{-})
\]
where $(\A,\B)\mapsto\dgFun{\A}{\B}$ denotes the passage to the dg category of
dg functors $\A\to\B$~\cite[Sec.~2.3]{Kel06} (see
also~\cite[Digression~7.3.1]{Rie14}). For example, for a small {dg category $\A$},
\[
  \dgMod{\A}\coloneqq\dgFun{\A^\op}{\dgMod[\UU]{\kk}}
\]
is the dg category of (right) dg $\A$-modules discussed in
\Cref{subsubsec:dg-modules}.

\begin{remark}
  The above closed symmetric monoidal structure restricts to the full
  subcategory $\dgcat\subseteq\dgCat$.
\end{remark}

\subsubsection{Differential graded categories up to quasi-equivalence}
\label{subsec:dg_cats_up_to_quasi-eq}

We say that a dg functor ${F\colon\A\to\B}$ is a \emph{quasi-equivalence} if the
induced functor of graded categories
\[
  \H{F}\colon\H{\A}\longrightarrow\H{\B}
\]
is an equivalence; explicitly, $F$ is a quasi-equivalence
if the functor of ordinary categories
\[
  \H[0]{F}\colon\H[0]{\A}\longrightarrow\H[0]{\B}
\]
is essentially surjective (=dense) and the morphisms of dg $\kk$-modules
\[
  F\colon\A(a_1,a_2)\longrightarrow\B(F(a_1),F(a_2)),\qquad a_1,\ a_2\in\A,
\]
are quasi-isomorphisms. We say that $F$ is a \emph{quasi-isomorphism} if it is
bijective on objects and a quasi-equivalence.

\begin{remark}
  A quasi-equivalence $F\colon\A\to\B$ induces a quasi-equivalence
  \[
    \mathbb{L}F_!\colon\dgDerCat{\A}\longrightarrow\dgDerCat{\B}
  \]
  between the corresponding derived dg categories~\cite[Sec.~3.8]{Kel06}.
\end{remark}

The category $\dgCat$ admits a Quillen model category structure whose weak equivalences are the
quasi-equivalences, which is typically called the Tabuada model
structure~\cite{Tab05a}. We denote the homotopy category of this model structure by\footnote{The membership relation $\UU\in\VV$ induces a
  fully faithful functor $\Hqe_{\UU}\to\Hqe_{\VV}=\Hqe$ that preserves
  Dwyer--Kan mapping spaces up to weak homotopy equivalence,
  see~\cite[p.~624]{Toe07}. Thus, we may regard the homotopy category of small
  dg category as a full subcategory of that of large categories without issue.}
\[
  \Hqe\coloneqq\dgCat[\mathrm{qeq}^{-1}].
\]
Thus, two dg categories are isomorphic in $\Hqe$ if and only if they are
connected by a zig-zag of quasi-equivalences.\footnote{Similar to
  quasi-isomorphisms, a quasi-equivalence need not admit an inverse even up to
  homotopy.} The tensor product with a dg category that is cofibrant in the Tabuada model
structure preserves quasi-equivalences, and hence the localisation $\Hqe$
inherits a symmetric monoidal structure $(\Hqe,\Lotimes)$~\cite[Sec.~4]{Toe07}.

\subsubsection{To{\"e}n's internal $\operatorname{Hom}$-functor}
\label{subsec:Toen}

Remarkably, the symmetric monoidal structure on $\Hqe$ is
closed~\cite[Thm.~6.1]{Toe07}: There exists a bifunctor, \emph{To{\"en}'s
  internal $\operatorname{Hom}$ functor},
\[
  \RHom{-}{-}\colon\Hqe^\op\times\Hqe\longrightarrow\Hqe
\]
equipped with natural bijections
\begin{equation*}
  \Hom[\Hqe]{\A\Lotimes\B}{\C}\stackrel{\sim}{\longrightarrow}\Hom[\Hqe]{\A}{\RHom{\B}{\C}},\qquad\A,\B,\C\in\Hqe,
\end{equation*}
and consequently also natural isomorphisms in $\Hqe$
\begin{equation}
  \label{eq:RHom-adjunction}
  \RHom{\A\Lotimes\B}{\C}\stackrel{\sim}{\longrightarrow}\RHom{\A}{\RHom{\B}{\C}},\qquad\A,\B,\C\in\Hqe,
\end{equation}
see for example~\cite[Rmk.~3.3.9]{Rie14}. In particular, there is a canonical
bijection~\cite[Lemma~3.4.9]{Rie14}
\[
  \Hom[\Hqe]{\kk}{\RHom{\A}{\B}}\stackrel{\sim}{\longleftrightarrow}\Hom[\Hqe]{\A}{\B},
\]
so that the isomorphism classes of objects in $\H[0]{\RHom{\A}{\B}}$ are in
bijection with the set of morphisms from $\A$ to $\B$ in
$\Hqe$~\cite[Cor.~4.9]{Toe07}. Finally, each morphism ${M\in\Hom[\Hqe]{\A}{\B}}$
induces a functor
\[
  F_M\colon\H[0]{\A}\longrightarrow\H[0]{\B}
\]
that is unique up to (non-unique) isomorphism, and which is an equivalence
whenever $M$ is an isomorphism in $\Hqe$, see~\cite[Sec.~2]{Toe07}. For this
reason, the objects of the dg category $\RHom{\A}{\B}$ are sometimes called
\emph{quasi-functors} (see for example~\cite[Sec.~4.1]{Kel06}).

The following properties of To{\"e}n's internal $\operatorname{Hom}$-functor are
of particular interest in the sequel (see~\Cref{subsec:consequences}):
\begin{itemize}
\item \cite[p.~649]{Toe07} There is a canonical isomorphism in $\Hqe$
  \begin{equation}
    \label{eq:dgDerCat-as-Hom}
    \RHom{\A^\op}{\dgDerCat{\kk}}\stackrel{\sim}{\longrightarrow}\dgDerCat{\A},\qquad\A\in\dgcat.
  \end{equation}
\item \cite[Thm.~7.2]{Toe07} Let $\A$ and $\B$ be small dg categories.
  Restriction along the dg Yoneda embedding ${\h\colon\A\to\dgDerCat{\A}}$
  induces an isomorphism in $\Hqe$
  \begin{equation}
    \label{eq:dg-Kans-theorem}
    \h^*\colon\RHom[c]{\dgDerCat{\A}}{\dgDerCat{\B}}\stackrel{\sim}{\longrightarrow}\RHom{\A}{\dgDerCat{\B}},
  \end{equation}
  where
  \[
    \RHom[c]{\dgDerCat{\A}}{\dgDerCat{\B}}\subseteq\RHom{\dgDerCat{\A}}{\dgDerCat{\B}}
  \]
  is the full dg subcategory spanned by the quasi-functors whose induced functor
  $\DerCat{\A}\longrightarrow\DerCat{\B}$ preserves small
  coproducts.\footnote{The subscript `$c$' in $\RHom[c]{-}{-}$ stands for
    `cocontinuous.'} For example, there are isomorphisms in $\Hqe$
  \[
    \dgDerCat{\A^\op}\stackrel{\eqref{eq:dgDerCat-as-Hom}}{\cong}\RHom{\A}{\dgDerCat{\kk}}\stackrel{\eqref{eq:dg-Kans-theorem}}\cong\RHom[c]{\dgDerCat{\A}}{\dgDerCat{\kk}}.
  \]
\item \cite[Cor.~7.6]{Toe07} The derived Eilenberg--Watts Theorem holds in
  $\Hqe$: For small dg categories $\A$ and $\B$, there is a canonical
  isomorphism in $\Hqe$
  \begin{equation}
    \label{eq:derived-EW-theorem}
    \dgDerCat{\A^\op\Lotimes\B}\stackrel{\sim}{\longrightarrow}\RHom[c]{\dgDerCat{\A}}{\dgDerCat{\B}}.
  \end{equation}
\item Let $\A$ and $\B$ be small dg categories. Since the monoidal structure on
  $\Hqe$ is symmetric and the passage to derived dg categories preserves quasi-equivalences, there are canonical isomorphisms in $\Hqe$
  \begin{equation}
    \label{eq:symmetry-bimods}
    \begin{tikzcd}[row sep=normal,column sep=normal]
      \RHom[c]{\dgDerCat{\A^\op}}{\dgDerCat{\B}}\rar[leftrightarrow,dotted]{\sim}&\RHom[c]{\dgDerCat{\B^\op}}{\dgDerCat{\A}}\\
      \dgDerCat{\A\Lotimes\B}\uar{\wr}\rar{\sim}&\dgDerCat{\B\Lotimes\A}\uar{\wr}
    \end{tikzcd}
  \end{equation}
  where the top vertical isomorphisms are instances of the derived
  Eilenberg--Watts Theorem~\eqref{eq:derived-EW-theorem} and the bottom vertical
  arrows are instances of the isomorphism~\eqref{eq:dgDerCat-as-Hom}.
  Consequently, since equivalences of categories preserve small coproducts, a
  pair of isomorphisms
  \[
    \dgDerCat{\A_1}\stackrel{\sim}{\to}\dgDerCat{\A_2}\qquad\text{and}\qquad
    \dgDerCat{\B_1}\stackrel{\sim}{\to}\dgDerCat{\B_2}
  \]
  in $\Hqe$ induce an isomorphism in $\Hqe$
  \begin{equation}
    \label{eq:-derived-invariance-bimodules}
    \dgDerCat{\A_1\Lotimes\B_1}\stackrel{\sim}{\longrightarrow}\dgDerCat{\A_2\Lotimes\B_2},
  \end{equation}
  obtained explicitly as the composite
  \begin{align*}
    \dgDerCat{\A_1\Lotimes\B_1}&\stackrel{\sim}{\to}\RHom[c]{\dgDerCat{\A_1^\op}}{\dgDerCat{\B_1}}&&\eqref{eq:derived-EW-theorem}\\
                               &\stackrel{\sim}{\to}\RHom[c]{\dgDerCat{\A_1^\op}}{\dgDerCat{\B_2}}\\
                               &\stackrel{\sim}{\to}\RHom[c]{\dgDerCat{\B_2^\op}}{\dgDerCat{\A_1}}&&\eqref{eq:symmetry-bimods}\\
                               &\stackrel{\sim}{\to}\RHom[c]{\dgDerCat{\B_2^\op}}{\dgDerCat{\A_2}}\\
                               &\stackrel{\sim}{\to}\RHom[c]{\dgDerCat{\A_2^\op}}{\dgDerCat{\B_2}}&&\eqref{eq:symmetry-bimods}\\
                               &\stackrel{\sim}{\leftarrow}\dgDerCat{\A_2\Lotimes\B_2};&&\eqref{eq:derived-EW-theorem}
  \end{align*}
  here the second and the fourth isomorphisms stem from the functoriality of
  To{\"e}n's internal $\operatorname{Hom}$.
\end{itemize}

\begin{remark}
  \label{rmk:alternative_means}
  The isomorphism in \eqref{eq:-derived-invariance-bimodules} can established
  through alternative means as in~\cite[Theorem~2.1]{Ric91}, which deals with
  ordinary algebras, using for example the results
  in~\cite[Sections~7--9]{Kel94}. Alternatively, the isomorphism in question
  also follows from the existence of a suitable symmetric monoidal structure on
  $\Hmo$, the homotopy category of dg categories considered up to Morita
  equivalence, see for example~\cite[Section~1.6]{Tab15}. In particular, the use
  of To{\"e}n's internal $\operatorname{Hom}$ is not strictly necessary to prove
  \Cref{thm:Q-shaped-derived-invariance,thm:independence-of-generators} (but see
  \Cref{rmk:Deligne-type-tensor-product,rmk:Q-shaped-LuriesTP} for motivation
  for its use).
\end{remark}

\begin{remark}
  \label{rmk:Deligne-type-tensor-product}
  The construction
  \[
    (\dgDerCat{\A},\dgDerCat{\B})\mapsto\LTP{\dgDerCat{\A}}{\dgDerCat{\B}}\coloneqq\dgDerCat{\A\Lotimes\B}
  \]
  can be regarded as a Deligne-type tensor product (compare with~\cite[Prop.~5.3]{Del90})
  of compactly-generated pre-triangulated dg categories in the following sense:
  Let $\A,\B,\C$ be small dg categories, then there are isomorphisms in $\Hqe$
  \[
    \RHom[c]{\LTP{\dgDerCat{\A}}{\dgDerCat{\B}}}{\dgDerCat{\C}}\cong\RHom[c]{\dgDerCat{\A}}{\RHom[c]{\dgDerCat{\B}}{\dgDerCat{\C}}}.
  \]
  Indeed, there are isomorphisms in $\Hqe$
  \begin{align*}
    \RHom[c]{\LTP{\dgDerCat{\A}}{\dgDerCat{\B}}}{\dgDerCat{\C}}&\stackrel{\phantom{\eqref{eq:dg-Kans-theorem}}}{=}\RHom[c]{\dgDerCat{\A\Lotimes \B}}{\dgDerCat{\C}}\\
                                                               &\stackrel{\eqref{eq:dg-Kans-theorem}}{\cong}\RHom{\A\Lotimes\B}{\dgDerCat{\C}}\\
                                                               &\stackrel{\eqref{eq:RHom-adjunction}}{\cong}\RHom{\A}{\RHom{\B}{\dgDerCat{\C}}}\\
                                                               &\stackrel{\eqref{eq:dg-Kans-theorem}}{\cong}\RHom{\A}{\RHom[c]{\dgDerCat{\B}}{\dgDerCat{\C}}}\\
                                                               &\stackrel{\eqref{eq:derived-EW-theorem}}{\cong}\RHom{\A}{\dgDerCat{\B^\op\Lotimes\C}}\\
                                                               &\stackrel{\eqref{eq:dg-Kans-theorem}}{\cong}\RHom[c]{\dgDerCat{\A}}{\dgDerCat{\B^\op\Lotimes\C}}\\
                                                               &\stackrel{\eqref{eq:derived-EW-theorem}}{\cong}\RHom[c]{\dgDerCat{\A}}{\RHom[c]{\dgDerCat{\B}}{\dgDerCat{\C}}}.
  \end{align*}
  The more general case of well-generated pre-triangulated dg categories is
  treated in~\cite{LRG22}.
\end{remark}

\begin{remark}
  Let $\A$ and $\B$ be two small dg categories. Notice that, in general, the dg
  categories $\dgDerCat{\A\Lotimes\B}$ and $\dgDerCat{\A}\Lotimes\dgDerCat{\B}$
  are not isomorphic in $\Hqe$. Indeed, the category
  $\H[0]{\dgDerCat{\A}\Lotimes\dgDerCat{\B}}$ is not triangulated in general
  (it may be missing some cones).
\end{remark}

\subsection{Keller's Recognition Theorem}
\label{subsec:Keller}

Let $(\E,\S)$ be a Frobenius exact category with small coproducts. By a
well-known theorem of Happel~\cite[Sec.~I.2]{Hap88}, the stable category
$\stable{\E}[\S]$ is a triangulated category with small coproducts (see
also~\cite[Sec.~4.3]{Kel94}). Suppose that there exists a set
$G\subseteq\stable{\E}[\S]$ of compact generators. Choose a complete
$\S$-projective resolutions
\[
  \begin{tikzcd}[row sep=tiny,column sep=small] \P{X}\colon& \cdots\drar[two
    heads]\ar{rr}{d_{\P{X}}^{-1}}&&\P{X}[0]\ar{rr}{d_{\P{X}}^{0}}\ar[two
    heads]{dr}[swap]{p^0}&&\P{X}[1]\ar{rr}{d_{\P{X}}^1}\drar[two
    heads]&&\cdots\\
    &&\Loops{X}\urar[tail]&&X\ar[tail]{ur}&&\Loops[-1]{X}\urar[tail]
  \end{tikzcd}
\]
for each object $X\in G$; thus $\P{X}$ is an acyclic cochain complex of
$\S$-projective objects in $\E$ such that
$\BH[0]{\P{X}}=X$.\footnote{In~\cite{Kel94} the convention is rather that
  $\Z[0]{\P{X}}=\BH[-1]{\P{X}}=X$. This discrepancy is inconsequential since the
  shift of cochain complexes is an automorphism. We prefer the adopted
  convention since it results in simpler formulas in our context.} More
generally, set
\[
  \Loops[i]{X}\coloneqq\BH[-i]{\P{X}},\qquad i\in\ZZ.
\]

We let $\dgCh{\E}$ be the dg category of cochain complexes in $\E$, with the
corresponding dg $\kk$-modules of morphisms defined analogously to the dg
$\kk$-module of morphisms between two graded $\kk$-modules
(see~\eqref{eq:culprits} and~\eqref{eq:partial}). Consider the full dg
subcategory
\[
  \dgP{G}\coloneqq\set{\P{X}}[X\in G]\subseteq\dgCh{\E}.
\]
Keller's Recognition Theorem~\cite[Sec.~4.3]{Kel94} states that there are
equivalences of (dg enhanced) triangulated categories
\begin{equation}
  \label{eq:Kellers-recognition}
  \begin{tikzcd}[row sep=tiny]
    \DerCat{\dgP{G}}\rar[leftarrow]{\sim}&\H[0]{\dgCh[\acycl]{\Proj{\E,\S}}}\rar{\sim}&\stable{\E}[\S]
  \end{tikzcd}
\end{equation}
between the derived category of the small dg category $\dgP{G}$ and the stable
category of the Frobenius exact category $(\E,\S)$; here
\[
  \dgCh[\acycl]{\Proj{\E,\S}}=
  \dgCh[(\E,\S)\text{-}\acycl]{\Proj{\E,\S}}\subseteq\dgCh{\E}
\]
denotes the full dg subcategory spanned by the acyclic cochain complexes of
$\S$-projective objects in $\E$.

For later use, we introduce the morphism of cochain complexes of $\kk$-modules
(compare with~\eqref{eq:Kellers-recognition})
\begin{equation}
  \label{eq:stableHom-dg}
  \psi=\psi_{X,Y}\colon\dgHom[\E]{\P{X}}{\P{Y}}\longrightarrow\dgHom[\E]{\P{X}}{Y},\qquad X,Y\in G,
\end{equation}
whose component of degree $-i$ is the morphism
of $\kk$-modules
\begin{equation}
  \label{eq:def-psi}
  \dgHom[\E]{\P{X}}{\P{Y}}[-i]\longrightarrow\Hom[\E]{\P{X}[i]}{Y},\quad
  (f^j)_{j\in\ZZ}\longmapsto (\P{X}[i]\stackrel{f^{i}}{\to}
  \P{Y}[0]\stackrel{p^0}{\to} Y);
\end{equation}
here, we view $Y$ as a cochain complex concentrated in degree $0$ and observe
the differentials on the target cochain complex are given by
$(-1)^id_{\P{X}}^i$.
Although the following lemma is well known to experts, we sketch a proof as it
plays an essential role in our proof of \Cref{thm:Q-shaped-bimodules}.

\begin{lemma}
  \label{lemma:psi}
  The morphism of dg $\kk$-modules~\eqref{eq:stableHom-dg}
  \[
    \psi\colon\dgHom[\E]{\P{X}}{\P{Y}}\longrightarrow\dgHom[\E]{\P{X}}{Y},
  \]
  is a quasi-isomorphism. In particular, $\psi$ induces an isomorphism of
  $\kk$-modules\footnote{This isomorphism is perhaps more familiar under the
    identification
    \[
      \sHom[\E,\S]{\Loops[i]{X}}{Y}\cong\sHom[\E,\S]{X}{\Suspension[i]{Y}},
    \]
    where $\Sigma\coloneqq\Omega^{-1}$ is the suspension functor in the
    triangulated category $\stable{\E}[\S]$.}
  \[
    \H[i]{\psi}\colon\H[i]{\dgHom[\E]{\P{X}}{\P{Y}}}\stackrel{\sim}{\longrightarrow}\H[i]{\dgHom[\E]{\P{X}}{Y}}\cong\sHom[\E,\S]{\Loops[i]{X}}{Y},
  \]
  for each $i\in\ZZ$, where the right-hand side denotes the $\kk$-module of
  morphisms $\Loops[i]{X}\to Y$ in the stable category $\stable{\E}[\S]$.
\end{lemma}
\begin{proof}
  We begin by observing that $\psi$ is a morphism of graded $\kk$-modules since
  its components
  \[
    \dgHom[\E]{\P{X}}{\P{Y}}[-i]=\prod_{j\in\ZZ}\Hom[\E]{\P{X}[j]}{\P{Y}[j-i]}\stackrel{\pi_i}{\longrightarrow}\Hom[\E]{\P{X}[i]}{\P{Y}[0]}\stackrel{p^0\circ?}{\longrightarrow}\Hom[\E]{\P{X}[i]}{Y}
  \]
  are composites of morphisms of $\kk$-modules. We now verify that $\psi$ is a
  morphism of dg $\kk$-modules: For a morphism
  $f\in\dgHom[\E]{\P{X}}{\P{Y}}[-i]$ we compute\footnote{Alternatively, notice that
  \[
    \psi=p\circ?\colon\dgHom[\E]{\P{X}}{\P{Y}}\longrightarrow\dgHom[\E]{\P{X}}{Y}
  \]
  is the morphism of dg $\kk$-modules induced by the canonical morphism of cochain
  complexes ${p\colon \P{Y}\to Y}$ whose only non-zero component is ${p^0\colon
  \P{Y}[0]\to Y}$.}
  \begin{align*}
    \psi(\partial(f))&=\psi((d_{\P{Y}}^{j-i}\circ f^j-(-1)^if^{j+1}\circ d_{\P{X}}^j)_{j\in\ZZ})&&\eqref{eq:partial-coords}\\
                     &=p^0\circ(d_{\P{Y}}^{-1}\circ f^{i-1}-(-)^if^i\circ d_{\P{X}}^{i-1})&&\eqref{eq:def-psi}\\
                     &=(p^0\circ f^{i})\circ((-1)^{i-1} d_{\P{X}}^{i-1})&&p^0\circ d_{\P{Y}}^{-1}=0\\
                     &=\psi(f)\circ((-1)^{i-1}d_{\P{X}}^{i-1});&&\eqref{eq:def-psi}
  \end{align*}  
  here, in the second equality we use that
  the morphism $\partial(f)\in\dgHom[\E]{\P{X}}{\P{Y}}[1-i]$ is homogeneous of {degree $1-i$}.  
  
  We now prove that $\psi$ is a quasi-isomorphism. Fix $i\in\ZZ$. We prove first
  that the morphism of $\kk$-modules $\H[-i]{\psi}$ is surjective. Let
  $g\in\Hom[\E]{\P{X}[i]}{Y}$ be a morphism such that $g\circ
  d_{\P{X}}^{i-1}=0$, so that $g$ represents a class in the $(-i)$-th cohomology
  of the target of the morphism $\psi$. By the universal property of the cokernel, the
  lifting property of $\S$-projective objects, and the extension property of
  $\S$-injective objects, there exists a commutative diagram
  \[
    \begin{tikzcd}[row sep=tiny,column sep=small] \P{X}\colon&
      \cdots\drar[two heads]\ar{rr}{d_{\P{X}}^{i-1}}&&\P{X}[i]\ar[dotted]{dddd}{f^i}\ar{rr}{d_{\P{X}}^{i}}\ar[two heads]{dr}\ar{dddr}[description]{g}&&\P{X}[i]\ar[dotted]{dddd}{(-1)^if^{i+1}}\ar{rr}{d_{\P{X}}^{i+1}}\drar[two heads]&&\cdots\\
      &&\cdots\urar[tail]&&\Loops[-i]{X}\ar[dotted]{dddr}\ar[tail]{ur}\ar[dotted]{dd}&&\cdots\urar[tail]\\
      \phantom{X}\\
      &&\cdots\drar[tail]&&Y\ar[tail]{dr}&&\cdots\drar[tail]\\
      \P{Y}\colon&\cdots\urar[two
      heads]\ar{rr}[swap]{d_{\P{Y}}^{-1}}&&\P{Y}[0]\ar{rr}[swap]{d_{\P{Y}}^{0}}\urar[two
      heads]{p^0}&&\P{Y}[1]\urar[two heads]\ar{rr}[swap]{d_{\P{Y}}^{1}}&&\cdots
    \end{tikzcd}
  \]
  In particular $d_{\P{Y}}^0\circ f^i=(-1)^i f^{i+1}\circ d_{\P{X}}^i$. By the
  usual inductive argument using lifting and extension properties of $\S$-projective
  and of $\S$-injective objects, the morphisms
  $f^i$ and $f^{i+1}$ extend to a degree $-i$ cocycle
  ${f\in\dgHom[\E]{\P{X}}{\P{Y}}[-i]}$, that is
  \[
    0=\partial(f)\stackrel{\eqref{eq:partial-coords}}{=}(d_{\P{Y}}^{j-i}\circ
    f^j-(-1)^if^{j+1}\circ d_{\P{X}}^j)_{j\in\ZZ},
  \]
  that by construction satisfies $\psi(f)=p^0\circ f^i=g$.
 
  We now prove that $\H[-i]{\psi}$ is injective. For this, let
  $f\in\dgHom[\E]{\P{X}}{\P{Y}}[-i]$ be a degree $-i$ cocycle such that $\psi(f)\in\Hom[\E]{\P{X}[i]}{Y}$ is
  cohomologically trivial, which is to say that
  \[
    \psi(f)=p^0\circ
    f^i\stackrel{!}{=}\widetilde{h}\circ(-1)^id_{\P{X}}^{i}=(-1)^i\widetilde{h}\circ
    d_{\P{X}}^{i}
  \]
  for some morphism $\widetilde{h}\colon\P{X}[i+1]\to Y$. The lifting property
  of $\S$-projective objects yields a diagram
  \[
    \begin{tikzcd}[row sep=tiny,column sep=small] \P{X}\colon\ar{dddd}{f}&
      \cdots\ar{dddd}\drar[two
      heads]\ar{rr}{d_{\P{X}}^{i-1}}&&\P{X}[i]\ar{dddd}{f^i}\ar{rr}{d_{\P{X}}^{i}}\ar[two
      heads]{dr}&&\P{X}[i+1]\ar{dddl}[description]{\widetilde{h}}\ar{dddd}{f^{i+1}}\ar[dotted]{ddddll}[description]{h^{i+1}}\ar{rr}{d_{\P{X}}^{i+1}}\drar[two
      heads]&&\cdots\\
      &&\cdots\urar[tail]&&\Loops[-i]{X}\ar[tail]{ur}&&\cdots\urar[tail]\\
      \phantom{X}\\
      &&\cdots\drar[tail]&&Y\ar[tail]{dr}&&\cdots\drar[tail]\\
      \P{Y}\colon&\cdots\urar[two
      heads]\ar{rr}[swap]{d_{\P{Y}}^{-1}}&&\P{Y}[0]\ar{rr}[swap]{d_{\P{Y}}^{0}}\ar[two
      heads]{ur}[description]{p^0}&&\P{Y}[1]\urar[two
      heads]\ar{rr}[swap]{d_{\P{Y}}^{1}}&&\cdots
    \end{tikzcd}
  \]
  such that $\widetilde{h}=p^0\circ h^{i+1}$. A straightforward inductive
  argument using lifting and extension properties shows that $h^{i+1}$ extends
  to a morphism ${h\in\dgHom[\E]{\P{X}}{\P{Y}}[-i-1]}$ such that
  \[
    f=\partial(h)\stackrel{\eqref{eq:partial-coords}}=(d_{\P{Y}}^{j-i-1}\circ
    h^j+(-1)^{i}h^{j+1}\circ d_{\P{X}}^j)_{j\in\ZZ}.
  \]
  For instance,
  \begin{equation}
    \label{eq:aux-psi}
    p^0\circ(f^i-(-1)^{i}h^{i+1}\circ d_{\P{X}}^{i})=p^0\circ
    f^i-(-1)^i\widetilde{h}\circ d_{\P{X}}^i=0
  \end{equation}
  so that $f^i-(-1)^{i}h^{i+1}\circ d_{\P{X}}^{i}$ factors through the kernel of
  $d_{\P{Y}}^0$. Equality~\eqref{eq:aux-psi} also implies that
  \[
    (-1)^if^{i+1}\circ d_{\P{X}}^i=d_{\P{Y}}^0\circ f^i=(-1)^id_{\P{Y}}^0\circ
    h^{i+1}\circ d_{\P{X}}^i
  \]
  so that $f^{i+1}-d_{\P{Y}}^0\circ h^{i+1}$ factors through the cokernel of
  $d_{\P{X}}^i$, etc. This shows that $\psi$ is a quasi-isomorphism of dg
  $\kk$-modules.
  
  Finally, the existence of isomorphisms of $\kk$-modules
  \[
    \H[-i]{\dgHom[\E]{\P{X}}{Y}}\cong\sHom[\E,\S]{\Loops[-i]{X}}{Y},\qquad
    i\in\ZZ,
  \]
  follows readily from the universal property of cokernels and the extension
  property of $\S$-injective objects, since a morphism $\Loops[-i]{X}\to Y$
  factors through an $\S$-injective object if and only if it factors through the
  admissible monomorphism ${\Loops[-i]{X}\rightarrowtail\P{X}[i+1]}$:
  \[
    \begin{tikzcd}[row sep=small,column sep=normal]
      \cdots\ar{rr}{d_{\P{X}}^{i-2}}&&\P{X}[i-1]\rar{d_{\P{X}}^{i-1}}\ar[bend
      right]{ddrr}[swap]{0}&\P{X}[i]\ar{rr}{d_{\P{X}}^{i}}\ar[bend
      right,dotted]{ddr}\drar[two heads]&&\P{X}[i+1]\ar[bend left,dotted]{ddl}\ar{rr}{d_{\P{X}}^{i+1}}&&\cdots\\
      &&&&\Loops[-i]{X}\urar[tail]\dar\\
      &&&&Y
    \end{tikzcd}
  \]
  This finishes the proof of the lemma.
\end{proof}

\subsection{$Q$-shaped derived categories}
\label{subsec:Q-shaped}

The following setup is fixed throughout this subsection, as it is needed in
order to apply~\cite[Thm.~D]{HJ24}, which yields the fact that the $Q$-shaped
derived category of a $\kk$-algebra $A$ is a compactly generated (algebraic)
triangulated $\kk$-category with a distinguished set of compact generators.

\begin{setup}
  \label{setup:Q-shaped-dercats}
  From now on the ground commutative ring $\kk$ is assumed to be a hereditary
  noetherian ring (e.g.~$\kk=\ZZ$ or $\kk$ is a field). We fix an arbitrary
  $\kk$-algebra $A$ and a small $\kk$-category $Q$ that satisfies the
  assumptions in~\cite[Setup~2.9]{HJ24}. We do not recall all of them in detail
  here, but only mention that these include the assumption that the
  $\kk$-modules of morphisms in $Q$ are finitely generated and projective, the
  existence of a Serre functor for $Q$, and the assumption that the category $Q$
  locally bounded and skeletal~\cite[Rmk.~7.6]{HJ22} (see
  also~\Cref{ex:stableQMod} for an important special case in which the
  assumptions are satisfied).
\end{setup}

\subsubsection{Construction of $\DerCat<Q>{A}$}

We recall the construction of $Q$-shaped derived categories following the
summary given in~\cite[Sec.~2.10]{HJ24}. Consider the $\kk$-category
\[
  \Mod{A\otimes Q}\coloneqq\Fun[\kk]{(A\otimes Q)^\op}{\Mod{\kk}}
\]
of right $(A\otimes Q)$-modules,\footnote{We warn the reader that
  \cite{HJ22,HJ24} work with \emph{left} modules---the discrepancy in
  conventions is immaterial since the assumptions on $Q$ are
  self-dual~\cite[Rmk.~2.6]{HJ22}. Furthermore, in \emph{op.~cit.}~the authors
  work with the category of $\kk$-linear functors $Q\to\Mod*{A}$, which is
  isomorphic to the category of left $(A\otimes Q)$-modules. We prefer to work
  with right $(A\otimes Q)$-modules since we wish to utilise the (covariant)
  Yoneda embedding in the proof of \Cref{thm:Q-shaped-bimodules}.} that is of
$\kk$-linear functors
\begin{align*}
  M\colon(A\otimes Q)^\op=(A^\op\otimes Q^\op)&\longrightarrow\Mod{\kk}\\
  q=(*,q)&\longmapsto M(q)=M(*,q).
\end{align*}
A morphism $f\colon M\to N$ between $(A\otimes Q)$-modules is therefore natural
transformation between $\kk$-linear functors, and hence it is given by a family
of morphisms of $\kk$-modules
\[
  f=(f_q\colon M(q)\to N(q))_{q\in Q}=(x \mapsto f_q(x))_{q\in Q}
\]
that satisfy the usual naturality constraints with respect to morphisms in
$A\otimes Q$.

The structure map $\kk\to A$ induces the standard adjunction
\[
  \begin{tikzcd}
    A\otimes-\colon\Mod{\kk}\rar[shift left]&\lar[shift
    left]\Mod{A}\noloc(-)^\natural,
  \end{tikzcd}
\]
where $M\mapsto M^\natural$ is the forgetful functor. By means of the
isomorphism of $\kk$-categories
\[
  \Mod{A\otimes Q}\cong\Fun[\kk]{Q^\op}{\Mod{A}},
\]
we obtain an induced adjunction (compare with~\cite[Def.~3.2 and
Cor.~3.5]{HJ22})
\[
  \begin{tikzcd}
    A\otimes-\colon\Mod{Q}\rar[shift left]&\lar[shift left]\Mod{A\otimes
      Q}\noloc(-)^\natural.
  \end{tikzcd}
\]
In particular, given a $Q$-module $M$, the $(A\otimes Q)$-module $A\otimes M$
acts on objects via
\[
  (A\otimes M)(q)\coloneqq A\otimes M(q),\qquad q\in Q,
\]
and on morphisms, utilising the right action of $A$ on itself, via
\begin{align}
  \label{eq:right-action-A}
  \begin{split}
    (A\otimes M)(b\otimes y)\colon A\otimes M(q)&\longrightarrow A\otimes
                                                  M(q')\\
    a\otimes m&\longmapsto ab\otimes M(y)(m),
  \end{split}
\end{align}
where $a,b\in A$ and $y\in Q(q',q)$.

Associated to the triple $(\kk,Q,A)$ there is the full subcategory
\begin{equation}
  \label{eq:exact-objects}
  \E\coloneqq\set{M\in\Mod{A\otimes Q}}[M^\natural\in\Mod{Q}\text{has finite
    projective dimension}]
\end{equation}
of \emph{exact objects}~\cite[Def.~4.1]{HJ22}. The subcategory
\begin{equation}
  \label{eq:semi-projectives}
  {}^{\perp_1}\E\coloneqq\set{M\in\Mod{A\otimes Q}}[\operatorname{Ext}_{A\otimes
    Q}^1(M,\E)=0]
\end{equation}
of \emph{semi-projective objects}\footnote{See~\cite[Ex.~3.3]{HJ24} for
  motivation for the use of this terminology.} is a Frobenius exact category
(with the exact structure inherited as an extension-closed subcategory of
$\Mod{A\otimes Q}$) and whose class of projective-injective objects is precisely
the class $\Proj{A\otimes Q}$ of projective $(A\otimes
Q)$-modules~\cite[Thm.~6.5]{HJ22}.

\begin{definition}[{\cite{HJ22,HJ24}}] The \emph{$Q$-shaped derived category of
    $A$} is by definition the (triangulated) stable category
  \[
    \DerCat<Q>{A}\coloneqq\underline{{}^{\perp_1}\E}.
  \]
\end{definition}

\begin{notation}
  We call the dg category
  \[
    \dgDerCat<Q>{A}\coloneqq\dgCh[{}^{\perp_1}\E\text{-}\acycl]{\Proj{A\otimes
        Q}}
  \]
  the \emph{$Q$-shaped derived dg category of $A$}. Notice that there is an
  equivalence of (dg enhanced) triangulated
  categories~\eqref{eq:Kellers-recognition}
  \[
    \H[0]{\dgDerCat<Q>{A}}\stackrel{\sim}{\longrightarrow}\DerCat<Q>{A},
  \]
  so that we may regard $\dgDerCat<Q>{A}$ as a canonical dg
  enhancement~\cite{BK90} of the $Q$-shaped derived category $\DerCat<Q>{A}$.
\end{notation}

\subsubsection{The canonical compact generators of $\DerCat<Q>{A}$}

By construction, the $Q$-shaped derived category $\DerCat<Q>{A}$ is an algebraic
triangulated category with small coproducts and, moreover, it admits a
distinguished set of compact generators, which we now describe. For an object
$q\in Q$, consider the `stalk' $Q$-module defined by~\cite[Def.~7.9 and
Lemma~7.10]{HJ22}
\begin{align}
  \label{eq:stalks-Q}
  \begin{split}
    S_q(q')\coloneqq\begin{cases}
      \kk&q=q',\\
      0&q\neq q',
    \end{cases}\qquad
    S_q(x)\coloneqq\begin{cases}
      \lambda\cdot \id[q]&x-\lambda\cdot \id[q]\in\mathfrak{r}_q,\\
      0&\text{otherwise},
    \end{cases}
  \end{split}
\end{align}
where $Q(q,q)\cong\kk\cdot\id[q]\oplus\mathfrak{r}_q$ for a (fixed) choice of
\emph{pseudo-radical} $\mathfrak{r}_q\subset Q(q,q)$~\cite[Def.~7.3 and
Lemma~7.7]{HJ22}. Then, the set
\begin{equation}
  \label{eq:stalks}
  S_Q(A)\coloneqq\set{A\otimes S_q\in\Mod{A\otimes Q}}[q\in Q]
\end{equation}
consists of semi-projective objects~\cite[Prop.~3.2]{HJ24} and, moreover, yields
a set of compact generators for $\DerCat<Q>{A}$~\cite[Thm.~D]{HJ24}.

\begin{example}
  \label{ex:stableQMod}
  Let $\kk$ be a field. By assumption, $Q$ has finite-dimensional morphisms
  spaces, is locally bounded and has a Serre functor in the sense
  of~\cite{BK90}. Consequently, $Q$ is a self-injective $\kk$-category and the
  $Q$-shaped derived category
  \[
    \DerCat<Q>{\kk}=\sMod{Q}
  \]
  is the stable category of all $Q$-modules. Indeed, there are equalities
  \[
    \E=\Proj{Q}=\Inj{Q}
  \]
  in this case (compare with~\eqref{eq:exact-objects}
  and~\eqref{eq:semi-projectives}). Moreover, the stalk $Q$-modules $S_q$, $q\in
  Q$, are the simple tops of the representable $Q$-modules in this case. In
  particular, if $Q$ is the $\kk$-category associated to a basic and
  elementary\footnote{Recall that a finite-dimensional $\kk$-algebra $\Lambda$
    is elementary if the endomorphism algebras of the simple $\Lambda$-modules
    are isomorphic to $\kk$. This assumption is needed in order to fit squarely
    within the assumptions in~\cite[Setup~2.9]{HJ24}, compare with condition
    (Rad) in~\cite[Def.~2.3]{HJ19a}.} finite-dimensional $\kk$-algebra that is
  self-injective, then $Q$ satisfies the conditions
  in~\Cref{setup:Q-shaped-dercats} with the Jacobson radical as the
  pseudo-radical (see~\cite[Ex.~7.8]{HJ22}).
\end{example}

\subsubsection{Description of $\DerCat<Q>{A}$ as the derived category of a small
  dg category}
\label{subsubsec:stalks}

Let $q\in Q$ be an object. There is a complete projective resolution $\P{S_q}$
of the (semi-projective) stalk $Q$-module $S_q$ with the following
properties~\cite[Prop.~5.11]{HJ24}:
\begin{itemize}
\item \cite[Def.~5.3, Lemmas~3.4 and~5.8, eqs.~$\sharp$14 and~$\sharp$15]{HJ24}
  The component of cohomological degree $i\in\ZZ$ of $\P{S_q}$ is of the form
  \[
    \P{S_q}[i]=\bigoplus_{j=1}^{n_i}P_{q_{i,j}}\otimes Q(-,q_{i,j}),
  \]
  where the $P_{q_{i,j}}$ are finitely generated projective $\kk$-modules.
\item \cite[Def.~3.8, Props.~5.7 and~5.11]{HJ24} The short exact sequences of $Q$-modules
  \[
    0\to\Omega_{{}^{\perp_1}{\E}}^{1-i}(S_q)\to
    \P{S_q}[i]\to\Omega_{{}^{\perp_1}{\E}}^{-i}(S_q)\to 0,\qquad i\in\ZZ,
  \]
  are object-wise split, which is to say that evaluating the above sequence at
  an arbitrary object of $Q$ yields a split short exact sequence of
  $\kk$-modules.
\end{itemize}
Since $\kk$-linear functors preserve split short exact sequences of $\kk$-modules, the cochain
complex of $(A\otimes Q)$-modules
\begin{equation}
  \label{eq:proj-AQ}
  \P{A\otimes S_q}\coloneqq A\otimes\P{S_q}
\end{equation}
is exact and is a complete projective resolution of the $(A\otimes Q)$-module
$A\otimes S_q$. For later use, we observe that $\P{A\otimes S_q}$ has the
components
\begin{equation}
  \label{eq:components-tensor-res}
  \P{A\otimes S_q}[i]=A\otimes\P{S_q}[i]\cong\bigoplus_{j=1}^{n_i}A\otimes P_{q_{i,j}}\otimes
  Q(-,q_{i,j}),\qquad i\in\ZZ,
\end{equation}
and the differential
\begin{equation}
  \label{eq:diff-APSq}
  d_{\P{A\otimes S_q}}=1_A\otimes d_{\P{S_q}}.
\end{equation}
Finally, we introduce the dg categories (compare with~\eqref{eq:stalks})
\begin{align}
  \label{eq:dg-stalks}
  \begin{split}
    \dgP{S_Q(\kk)}&\coloneqq\set{\P{S_q}}[q\in Q]\subseteq\dgCh{\Mod{Q}}\text{ and}\\
    \dgP{S_Q(A)}&\coloneqq\set{\P{A\otimes S_q}}[q\in Q]=\set{A\otimes\P{S_q}}[q\in Q]\subseteq\dgCh{\Mod{A\otimes Q}}.
  \end{split}
\end{align}
By Keller's Recognition Theorem, there are equivalences of (dg enhanced)
triangulated categories~\eqref{eq:Kellers-recognition}
\begin{equation}
  \label{eq:keller-Q}
  \DerCat{\dgP{S_Q(\kk)}}\stackrel{\sim}{\longleftrightarrow}\DerCat<Q>{\kk}\qquad\text{and}\qquad\DerCat{\dgP{S_Q(A)}}\stackrel{\sim}{\longleftrightarrow}\DerCat<Q>{A}.
\end{equation}

\section{Main results}
\label{sec:main-result}

\subsection{The proof of \Cref{thm:Q-shaped-bimodules}}

Recall \Cref{setup:Q-shaped-dercats} and the notations introduced
in~\Cref{subsec:Q-shaped}. To a $\kk$-module $V$ and a $Q$-module $M$ we
associate the $Q$-module
\[
  \begin{tikzcd}[column sep=large]
    \Hom[\kk]{V}{M}\colon Q^\op\rar{M}&\Mod{\kk}\rar{\Hom[\kk]{V}{-}}&\Mod{\kk}
  \end{tikzcd}
\]
and the $(A\otimes Q)$-module
\[
  \begin{tikzcd}[column sep=large]
    \Hom[\kk]{V}{A\otimes M}\colon (A\otimes Q)^\op\rar{A\otimes
      M}&\Mod{\kk}\rar{\Hom[\kk]{V}{-}}&\Mod{\kk}.
  \end{tikzcd}
\]
We need the following general observation (compare
with~\cite[Lemma~2.3.4]{GHJS24}).

\begin{lemma}
  \label{lemma:rho}
  Let $P$ be a finitely generated projective $\kk$-module and $M$ an arbitrary
  $Q$-module. Then, there is a natural isomorphism of $\kk$-modules
  \[
    A\otimes\Hom[Q]{P\otimes
      Q(-,q)}{M}\stackrel{\sim}{\longrightarrow}\Hom[A\otimes Q]{A \otimes
      P\otimes Q(-,q)}{A\otimes M},\quad q\in Q,
  \]
  given by the formula
  \[
    a\otimes f\mapsto(b \otimes x \otimes y\mapsto ab\otimes f_{v}(x \otimes
    y))_{v\in Q},\quad a\in A,\ f\in\Hom[Q]{P\otimes Q(-,q)}{M}.
  \]
\end{lemma}
\begin{proof}
  We define the required isomorphism as the composite of the natural
  isomorphisms of $\kk$-modules
  \begin{align*}
    A\otimes\Hom[Q]{P\otimes Q(-,q)}{M}&\cong A\otimes\Hom[Q]{Q(-,q)}{\Hom[\kk]{P}{M}}\\
                                       &\stackrel{\sim}{\to} A\otimes\Hom[\kk]{P}{M(q)}\\
                                       &\stackrel{\sim}{\to}\Hom[\kk]{P}{A\otimes M(q)}\\
                                       &\stackrel{\sim}{\leftarrow}\Hom[A\otimes Q]{(A\otimes Q)(-,q)}{\Hom[\kk]{P}{A\otimes M}}\\
                                       &\cong\Hom[A\otimes Q]{A\otimes P\otimes Q(-,q)}{A\otimes M}.
  \end{align*}
  Here, the third isomorphism is given by the tensor-evaluation
  morphism~\cite[Lemma~1.1]{Ish65}
  \begin{align}
    \label{eq:tensor-evaluation}
    \begin{split}
      A\otimes\Hom[\kk]{P}{M(q)}&\longrightarrow\Hom[\kk]{P}{A\otimes M(q)}\\
      a\otimes u&\longmapsto (x\mapsto a\otimes u(x)),
    \end{split}
  \end{align}
  which is an isomorphism due to the assumption that the $\kk$-module $P$ is
  finitely generated and projective. In order to obtain the explicit formula in
  the statement of the lemma, we observe that the fourth isomorphism is given by
  the inverse to the Yoneda bijection:
  \begin{align}
    \label{eq:inverse-yoneda}
    \begin{split}
      \Hom[\kk]{P}{A\otimes M(q)}&\longrightarrow\Hom[A\otimes Q]{(A\otimes Q)(-,q)}{\Hom[\kk]{P}{A\otimes M}}\\
      g&\longmapsto (b\otimes y\mapsto (x\mapsto (A\otimes M)(b\otimes y)(g(x))))_{v\in Q}.
    \end{split}
  \end{align}
  With these formulas in hand we compute, for $a\in A$ and $f\in\Hom[Q]{P\otimes
    Q(-,q)}{M}$,
  \begin{align*}
    a\otimes f&=a\otimes (x\otimes y\mapsto f_v(x\otimes y))_{v\in Q}\\
              &\mapsto a\otimes (y\mapsto(x\mapsto f_{v}(x\otimes y)))_{v\in Q}\\
              &\mapsto a\otimes(x\mapsto f_q(x\otimes\id[q]))&&\text{Yoneda bijection}\\
              &\mapsto (x\mapsto a\otimes f_q(x\otimes\id[q]))&&\eqref{eq:tensor-evaluation}\\
              &\mapsto (b\otimes y\mapsto (x\mapsto (A\otimes M)(b\otimes y)(a\otimes f_q(x\otimes\id[q])))_{v\in Q}&&\eqref{eq:inverse-yoneda}\\
              &\mapsto (b\otimes x\otimes y\mapsto (A\otimes M)(b\otimes y)(a\otimes f_q(x\otimes\id[q])))_{v\in Q}\\
              &=(b\otimes x\otimes y\mapsto ab\otimes M(y)(f_q(x\otimes\id[q]))_{v\in Q}&&\eqref{eq:right-action-A}\\
              &=(b\otimes x\otimes y\mapsto ab\otimes f_{v}(x\otimes y))_{v\in Q},&&\eqref{eq:inverse-yoneda2}
  \end{align*}
  where in the last step we use the identity
  \begin{equation}
    \label{eq:inverse-yoneda2}
    (y\mapsto(x\mapsto f_{v}(x\otimes y)))_{v\in Q}=(y\mapsto(x\mapsto
    M(y)(f_q(x\otimes\id[q]))))_{v\in Q}
  \end{equation}
  corresponding to the inverse of the Yoneda bijection
  \[
    \Hom[Q]{Q(-,q)}{\Hom[\kk]{P}{M}}\stackrel{\sim}{\longrightarrow}\Hom[\kk]{P}{M(q)},\qquad
    h\longmapsto h_q(\id[q]).
  \]
  This finishes the proof.
\end{proof}

We begin our work towards the proof of \Cref{thm:Q-shaped-bimodules} by
replacing $\dgP{S_Q(A)}$ by an isomorphic (!) dg category whose relationship to
the dg category $A\otimes\dgP{S_Q(\kk)}$ is more apparent (see
\Cref{defprop:dgG} and \Cref{prop:the-dg-functor}).

Recall the isomorphism~\eqref{eq:components-tensor-res}. For objects $q,q'\in Q$
and integers $i$ and $j$, \Cref{lemma:rho} provides an isomorphism of
$\kk$-modules
\begin{align*}
  A\otimes&\Hom[Q]{\P{S_q}[i]}{\P{S_{q'}}[j]}\\&\cong\bigoplus_{k=1}^{n_i}\bigoplus_{\ell=1}^{n_j}A\otimes\Hom[Q]{P_{q_{i,k}}\otimes Q(-,q_{i,k})}{P_{q_{j,\ell}'}\otimes Q(-,{q}_{j,\ell}')}\\
          &\stackrel{\sim}{\to}\bigoplus_{k=1}^{n_i}\bigoplus_{\ell=1}^{n_j}\Hom[A\otimes Q]{A\otimes P_{q_{i,k}}\otimes Q(-,q_{i,k})}{A\otimes P_{q_{j,\ell}'}\otimes Q(-,{q}_{j,\ell}')}\\
          &\cong\Hom[A\otimes Q]{A\otimes\P{S_q}[i]}{A\otimes\P{S_{q'}}[j]},
\end{align*}
given explicitly by the formula
\begin{equation}
  \label{eq:rho}
  \rho\colon a\otimes f\mapsto(b\otimes x\mapsto ab\otimes f_v(x))_{v\in Q}.
\end{equation}
Taking the product over $j\in\ZZ$ we obtain isomorphisms of $\kk$-modules
\begin{align*}
  \prod_{j\in\ZZ}A\otimes\Hom[Q]{\P{S_q}[j]}{\P{S_{q'}}[i+j]}&\stackrel{\sim}{\to}\prod_{j\in\ZZ}\Hom[A\otimes Q]{A\otimes\P{S_q}[j]}{A\otimes\P{S_{q'}}[i+j]}\\
                                                             &=\dgHom[A\otimes Q]{A\otimes\P{S_q}}{A\otimes\P{S_{q'}}}[i]\\
                                                             &=\dgHom[A\otimes Q]{\P{A\otimes S_q}}{\P{A\otimes S_{q'}}}[i]
\end{align*}
that assemble into an isomorphism of graded $\kk$-modules
\begin{equation}
  \label{eq:rho-graded}
  \rho\colon\coprod_{i\in\ZZ}\left(\prod_{j\in\ZZ}A\otimes\Hom[Q]{\P{S_q}[j]}{\P{S_{q'}}[i+j]}\right)\stackrel{\sim}{\longrightarrow}\dgHom[A\otimes
  Q]{\P{A\otimes S_q}}{\P{A\otimes S_{q'}}}.
\end{equation}
We endow the source of $\rho$ with the differential
\begin{align}
  \label{eq:transported-differential}
  \begin{split}
    \partial((a^j\otimes f^j))_{j\in\ZZ})
    &\coloneqq((1_A\otimes d_{\P{S_{q'}}}^{i+j})\circ(a^j\otimes f^j)\\&\phantom{====}-(-1)^i(a^{j+1}\otimes f^{j+1})\circ(1_A\otimes d_{\P{S_q}}^j))_{j\in\ZZ}\\
    &=(a^j\otimes(d_{\P{S_{q'}}}^{i+j}\circ
      f^j)-(-1)^ia^{j+1}\otimes(f^{j+1}\circ d_{\P{S_q}}^j))_{j\in\ZZ},
  \end{split}
\end{align}
where
\[
  (a^j\otimes
  f^j)_{j\in\ZZ}\in\prod_{j\in\ZZ}A\otimes\Hom[Q]{\P{S_q}[j]}{\P{S_{q'}}[i+j]}
\]
is a morphism of degree $i$.

\begin{lemma}
  \label{lemma:transported-differential}
  The following square commutes, where the vertical map on the left is defined
  by~\eqref{eq:transported-differential}:
  \[
    \begin{tikzcd}
      \prod_{j\in\ZZ}A\otimes\Hom[Q]{\P{S_q}[j]}{\P{S_{q'}}[i+j]}\rar{\rho}\dar{\partial}&\dgHom[A\otimes Q]{\P{A\otimes S_q}}{\P{A\otimes S_{q'}}}[i]\dar{\partial}\\
      \prod_{j\in\ZZ}A\otimes\Hom[Q]{\P{S_q}[j]}{\P{S_{q'}}[i+j+1]}\rar{\rho}&\dgHom[A\otimes
      Q]{\P{A\otimes S_q}}{\P{A\otimes S_{q'}}}[i+1]
    \end{tikzcd}
  \]
  Consequently, \eqref{eq:rho-graded} is an isomorphism of dg $\kk$-modules.
\end{lemma}
\begin{proof}
  On the one hand,
  \begin{align*}
    (a^j\otimes f^j)_{j\in\ZZ}&\stackrel{\rho}{\mapsto}((b\otimes x\mapsto a^jb\otimes f_v^j(x))_{v\in Q})_{j\in\ZZ}&&\eqref{eq:rho}\\
                              &\stackrel{\partial}{\mapsto}((b\otimes x\mapsto a^jb\otimes d_{\P{S_{q'}}}^{i+j}(f_v^j(x))&&\eqref{eq:partial},\ \eqref{eq:diff-APSq}\\
                              &\phantom{====}-(-1)^ia^{j+1}b\otimes f_v^{j+1}(d_{\P{S_q}}^j(x)))_{v\in Q})_{j\in\ZZ}.
  \end{align*}
  On the other hand,
  \begin{align*}
    (a^j\otimes f^j)_{j\in\ZZ}&\stackrel{\partial}{\mapsto}(a^j\otimes(d_{\P{S_{q'}}}^{i+j}\circ f^j-(-1)^ia^{j+1}\otimes f^{j+1}\circ d_{\P{S_q}}^j))_{j\in\ZZ}&&\eqref{eq:transported-differential}\\
                              &\stackrel{\rho}{\mapsto}((b\otimes x\mapsto a^jb\otimes d_{\P{S_{q'}}}^{i+j}(f_v^j(x))&&\eqref{eq:rho}\\
                              &\phantom{====}-(-1)^ia^{j+1}b\otimes f_v^{j+1}(d_{\P{S_q}}^j(x)))_{v\in Q})_{j\in\ZZ}.
  \end{align*}
  The lemma follows.
\end{proof}

Consider now the composition law given by the formula (compare
with~\eqref{eq:graded-composition})
\begin{align}
  \label{eq:transported-composition}
  \begin{split}
    (b^j\otimes g^j)_{j\in\ZZ}\circ(a^j\otimes f^j)_{j\in\ZZ}\coloneqq((b^{i+j}a^j)\otimes (g^{i+j}\circ f^j))_{j\in\ZZ},
  \end{split}  
\end{align}
where
\begin{align*}
  (a^j\otimes
  f^j)_{j\in\ZZ}&\in\prod_{j\in\ZZ}A\otimes\Hom[Q]{\P{S_q}[j]}{\P{S_{q'}}[i+j]}\intertext{and}(b^j\otimes
                  g^j)_{j\in\ZZ}&\in\prod_{j\in\ZZ}A\otimes\Hom[Q]{\P{S_{q'}}[j]}{\P{S_{q''}}[j+k]}
\end{align*}
are homogeneous elements of degree $i$ and $k$, respectively.

\begin{lemma}
  \label{lemma:transported-composition}
  The isomorphisms~\eqref{eq:rho-graded} are compatible with the composition
  law~\eqref{eq:transported-composition}.
\end{lemma}
\begin{proof}
  Indeed,
  \begin{align*}
    \rho((b^j\otimes g^j)_{j\in\ZZ}\circ(a^j\otimes f^j)_{j\in\ZZ})&\stackrel{\eqref{eq:transported-composition}}{=}\rho(((b^{i+j}a^j)\otimes (g^{i+j}\circ f^j))_{j\in\ZZ})\\&\stackrel{\eqref{eq:rho}}{=}((c\otimes x\mapsto b^{i+j}a^jc\otimes g_v^{i+j}(f_v^j(x)))_{v\in Q})_{j\in\ZZ},
  \end{align*}
  while
  \begin{align*}
    \rho((b^j\otimes g^j)_{j\in\ZZ})\circ\rho((a^j\otimes f^j)_{j\in\ZZ})=((c\otimes x&\stackrel{\eqref{eq:rho}}{\mapsto} a^jc\otimes f_v^j(x)\\&\stackrel{\eqref{eq:rho}}{\mapsto} b^{i+j}a^jc\otimes g_v^{i+j}(f_v^j(x)))_{v\in Q})_{j\in\ZZ}
  \end{align*}
  since
  \[
    \rho((a^j\otimes f^j)_{j\in\ZZ})\in\dgHom[A\otimes Q]{\P{A\otimes
        S_q}}{\P{A\otimes S_{q'}}}[i]
  \]
  is homogeneous of degree $i$, compare also with~\eqref{eq:graded-composition}.
\end{proof}

\begin{defprop}
  \label{defprop:dgG}
  Let $\dgG{S_Q(A)}$ be the dg category with the same objects as $\dgP{S_Q(A)}$
  and whose dg $\kk$-modules of morphisms have the homogeneous components
  $(i\in\ZZ)$
  \[
    \GdgHom[A\otimes Q]{A\otimes \P{S_q}}{A\otimes
      \P{S_{q'}}}[i]\coloneqq\prod_{j\in\ZZ}A\otimes\Hom[Q]{\P{S_q}[j]}{\P{S_{q'}}[i+j]},\quad
    q,q'\in Q.
  \]
  These graded $\kk$-modules are endowed with the differential defined
  by~\eqref{eq:transported-differential} and the composition law defined
  by~\eqref{eq:transported-composition}. Then, the isomorphisms of graded
  $\kk$-modules~\eqref{eq:rho-graded} are the components of an isomorphism of dg
  categories
  \[
    \rho\colon\dgG{S_Q(A)}\stackrel{\cong}{\longrightarrow}\dgP{S_Q(A)}
  \]
  that is the identity on the corresponding sets of objects.
\end{defprop}
\begin{proof}
  \Cref{lemma:transported-differential,lemma:transported-composition} show that
  $\dgG{S_Q(A)}$ is the dg category obtained by transporting the dg category
  structure of $\dgP{S_Q(A)}$ along the isomorphisms of graded
  $\kk$-modules~\eqref{eq:rho-graded}, in particular $\dgG{S_Q(A)}$ is a
  well-defined dg category.
\end{proof}

The following result provides an unconditional relationship between the dg
categories ${A\otimes\dgP{S_Q(\kk)}}$ and $\dgP{S_Q(A)}$ defined
in~\eqref{eq:dg-stalks}.

\begin{proposition}
  \label{prop:the-dg-functor}
  Recall~\Cref{setup:Q-shaped-dercats}. With the notations above, there is a dg
  functor
  \[
    A\otimes\dgP{S_Q(\kk)}\stackrel{\varphi}\longrightarrow\dgG{S_Q(A)}\stackrel{\sim}{\longrightarrow}\dgP{S_Q(A)}
  \]
  that is a bijection on the corresponding sets of objects and whose components
  \begin{equation}
    \label{eq:graded-map}
    \varphi\colon A\otimes\dgHom[Q]{\P{S_q}}{\P{S_{q'}}}\longrightarrow\GdgHom[A\otimes Q]{\P{A\otimes S_q}}{\P{A\otimes S_{q'}}},\quad q,q'\in Q,
  \end{equation}
  are given by the formula $(i\in\ZZ)$
  \begin{equation}
    \label{eq:graded-map-formula}
    a\otimes f\longmapsto(a\otimes
    f^j)_{j\in\ZZ},\qquad a\in A,\
    f\in\dgHom[Q]{\P{S_q}}{\P{S_{q'}}}[i].
  \end{equation}
\end{proposition}
\begin{proof}
  We begin by observing that the sets of objects of the dg categories
  $A\otimes\dgG{S_Q(\kk)}$ and $\dgP{S_Q(A)}$ are both in canonical bijection
  with the objects of $Q$. Consider the canonical morphisms of $\kk$-modules
  ($i\in\ZZ$)
  \begin{align}
    \label{eq:map-to-product}
    \begin{split}
      A\otimes\dgHom[Q]{\P{S_q}}{\P{S_{q'}}}[i]&=A\otimes\left(\prod_{j\in\ZZ}\Hom[Q]{\P{S_q}[j]}{\P{S_{q'}}[i+j]}\right)\\
                                               &\to\prod_{j\in\ZZ}A\otimes\Hom[Q]{\P{S_q}[j]}{\P{S_{q'}}[i+j]}\\
                                               &=\GdgHom[A\otimes Q]{\P{A\otimes S_q}}{\P{A\otimes S_{q'}}}[i].
    \end{split}
  \end{align}
  Notice that the above morphisms need not be isomorphisms since the functor
  $A\otimes-$ need not preserve infinite products of $\kk$-modules. Nonetheless,
  given that the functor $A\otimes-$ does preserve small coproducts of
  $\kk$-modules and that the algebra $A$ is concentrated in degree $0$, we
  obtain the required morphism of graded $\kk$-modules~\eqref{eq:graded-map}
  whose component of degree $i\in\ZZ$ is given by the formula
  in~\eqref{eq:graded-map-formula} for it is induced by the canonical
  projections from the product onto its factors (compare
  with~\Cref{ex:tensor-product-0}). We claim that these morphisms are in fact
  morphisms of dg $\kk$-modules. Indeed, the differential on the source of
  $\varphi$ is given by the formula (compare
  with~\Cref{ex:tensor-product-dg-cats-0})
  \begin{equation}
    \label{eq:diff-source-phi}
    (1_A\otimes\partial)\colon(a\otimes f)\longmapsto a\otimes\partial(f),\qquad
    a\in A,\ f\in\dgHom[Q]{\P{S_q}}{\P{S_{q'}}},
  \end{equation}
  where no Koszul sign appears since $A$ is concentrated in degree $0$ (compare
  with \Cref{ex:tensor-product-0}). For $a\in A$ and a homogeneous morphism
  \[
    f\in\dgHom[Q]{\P{S_q}}{\P{S_{q'}}}[i]=\prod_{j\in\ZZ}\Hom[Q]{\P{S_q}[j]}{\P{S_{q'}}[i+j]}
  \]
  of degree $i$ we compute
  \begin{align*}
    \varphi(a\otimes\partial(f))&=\varphi(a\otimes(d_{\P{S_{q'}}}\circ f-(-1)^i f\circ d_{\P{S_q}}))&&\eqref{eq:partial},\ \eqref{eq:diff-source-phi}\\
                                &=(a\otimes(d_{\P{S_{q'}}}^{i+j}\circ f^j-(-1)^if^{j+1}\circ d_{\P{S_q}}^j))_{j\in\ZZ}&&\eqref{eq:graded-map-formula}\\
                                &=(a\otimes(d_{\P{S_{q'}}}^{i+j}\circ f^j)- (-1)^ia\otimes(f^{j+1}\circ d_{\P{S_q}}^j))_{j\in\ZZ}\\
                                &=\partial((a\otimes f^j)_{j\in\ZZ})&&\eqref{eq:transported-differential}\\
                                &=\partial(\varphi(a\otimes f)).&&\eqref{eq:graded-map-formula}
  \end{align*}
  
  The morphisms \eqref{eq:graded-map} are also compatible with the corresponding
  composition laws: Given $a,b\in A$ and composable homogeneous morphisms
  \begin{align*}
    f&\in\dgHom[Q]{\P{S_q}}{\P{S_{q'}}}[i]=\prod_{j\in\ZZ}\Hom[Q]{\P{S_q}[j]}{\P{S_{q'}}[i+j]}\intertext{and}
       g&\in\dgHom[Q]{\P{S_{q'}}}{\P{S_{q''}}}[k]=\prod_{j\in\ZZ}\Hom[Q]{\P{S_{q'}}[j]}{\P{S_{q''}}[j+k]}
  \end{align*}
  with $f$ of degree $i$ we compute
  \begin{align*}
    \varphi((b\otimes g)\circ(a\otimes f))&=\varphi(ba\otimes(g\circ f))&&\eqref{eq:comp-tensor-product}\\
                                          &=\varphi(ba\otimes(g^{i+j}\circ f^j)_{j\in\ZZ})&&\eqref{eq:graded-composition}\\
                                          &=(ba\otimes(g^{i+j}\circ f^j))_{j\in\ZZ}&&\eqref{eq:graded-map-formula}\\
                                          &=((b\otimes g^{i+j})\circ (a\otimes f^j))_{j\in\ZZ}\\
                                          &=(b\otimes g^{j})_{j\in\ZZ}\circ (a\otimes f^j)_{j\in\ZZ}&&\eqref{eq:transported-composition}\\
                                          &=\varphi(b\otimes g)\circ\varphi(a\otimes f).&&\eqref{eq:graded-map-formula}
  \end{align*}
  As before, no Koszul signs arise in the previous computation since $A$ is
  concentrated in degree $0$ (compare with \Cref{ex:tensor-product-dg-cats-0}).
  
  Finally, keeping in mind the definition of the dg $\kk$-modules of morphisms
  in the tensor product of two dg categories
  (see~\eqref{eq:tensor-product-dg-k-mods}
  and~\eqref{eq:homs-in-tensor-product}), the above computations show that the
  morphisms \eqref{eq:graded-map} are the components of a dg functor
  ${\varphi\colon A\otimes\dgP{S_Q(\kk)}\to\dgG{S_Q(A)}}$.
\end{proof}

We are now ready to prove the main theorem in this article.

\begin{proof}[Proof of \Cref{thm:Q-shaped-bimodules}]
  Recall~\Cref{setup:Q-shaped-dercats}. In view of \Cref{defprop:dgG} and
  \Cref{prop:the-dg-functor}, to prove the theorem it suffices to show that the
  dg functor
  \[
    A\otimes\dgP{S_Q(\kk)}\stackrel{\varphi}{\longrightarrow}\dgG{S_Q(A)}\stackrel{\sim}{\longrightarrow}\dgP{S_Q(A)}
  \]
  constructed in the latter proposition is a quasi-isomorphism of dg categories,
  for this yields the desired equivalence of (dg enhanced) triangulated
  categories
  \[
    \DerCat{A\otimes\dgP{S_Q(\kk)}}\simeq\DerCat{\dgP{S_Q(A)}}\stackrel{\sim}{\longleftrightarrow}\DerCat<Q>{A},
  \]
  see~\eqref{eq:keller-Q}.

  Suppose first that the underlying $\kk$-module of the algebra $A$ is finitely
  generated. In this case the dg functor $\varphi$ is an isomorphism: The
  assumption that algebra $A$ is finitely generated as a $\kk$-module
  (equivalently, finitely presented since the ground commutative ring is assumed
  to be noetherian) implies that the functor $A\otimes-$ preserves small
  products (!) of $\kk$-modules\footnote{ Indeed, since infinite products of
    $\kk$-modules commute with finite products (=direct sums) of such, the
    functor $F\otimes-$ preserves infinite products for each finite-rank free
    $\kk$-module $F$; but then so does the functor $M\otimes-$ for each finitely
    presented $\kk$-module $M$ since this property is stable under the passage
    to cokernels for the tensor product functor is right exact in each variable
    separately and the infinite product of exact sequences of $\kk$-modules
    remains exact.} and hence the canonical maps~\eqref{eq:map-to-product} are
  isomorphisms.

  Suppose now that the underlying $\kk$-module of $A$ is flat but not
  necessarily finitely generated. We claim that the dg functor $\varphi$ is a
  quasi-isomorphism in this case. For this, given objects $q,q'\in Q$, we prove
  that the square below is commutative, where the quasi-isomorphisms $\psi^Q$
  and $\psi^{A\otimes Q}$ are instances of~\eqref{eq:stableHom-dg} (see
  also~\Cref{lemma:psi}):
  \begin{equation}
    \label{eq:key-square}
    \begin{tikzcd}
      A\otimes
      \dgHom[Q]{\P{S_q}}{\P{S_{q'}}}\rar{\varphi}\dar{A\otimes\psi^{Q}}&\GdgHom[A\otimes
      Q]{\P{A\otimes S_q}}{\P{A\otimes
          S_{q'}}}\dar{\psi^{A\otimes Q}}\\
      A\otimes\dgHom[Q]{\P{S_{q}}}{S_{q'}}\rar[equals]&A\otimes\dgHom[Q]{\P{S_{q}}}{S_{q'}}.
    \end{tikzcd}
  \end{equation}
  This immediately implies the claim since, by the $\kk$-flatness assumption on
  $A$, the functor $A\otimes-$ preserves quasi-isomorphisms of dg $\kk$-modules
  and hence the vertical morphism $A\otimes\psi^Q$ is a quasi-isomorphism.

  To prove that the square~\eqref{eq:key-square} commutes, we observe that,
  under the isomorphism of dg $\kk$-modules (see~\Cref{defprop:dgG})
  \[
    \rho\colon\GdgHom[A\otimes Q]{\P{A\otimes S_q}}{\P{A\otimes
        S_{q'}}}\stackrel{\sim}{\longrightarrow}\dgHom[A\otimes Q]{\P{A\otimes
        S_q}}{\P{A\otimes S_{q'}}},
  \]
  the degree $-i$ component of the morphism $\psi^{A\otimes Q}$ is given by the
  composite
  \begin{equation}
    \label{eq:transferred-psi}
    \begin{tikzcd}
      \GdgHom[A\otimes Q]{\P{A\otimes S_q}}{\P{A\otimes
          S_{q'}}}[-i]\rar{\pi_i}\dar[swap]{\psi^{A\otimes
          Q}}&A\otimes\Hom[Q]{\P{S_q}[i]}{\P{S_{q'}}[0]}\dar{1_A\otimes (p^0\circ?)}\\
      A\otimes\Hom[Q]{\P{S_q}[i]}{S_{q'}}\rar[equals]&A\otimes\Hom[Q]{\P{S_q}[i]}{S_{q'}}
    \end{tikzcd}
  \end{equation}
  More precisely, the following diagram commutes where the isomorphisms in the
  middle and bottom rows are instances of~\Cref{lemma:rho}:
  \[
    \begin{tikzcd}[column sep=small]
      \GdgHom[A\otimes Q]{\P{A\otimes S_q}}{\P{A\otimes
          S_{q'}}}\dar{\pi_i}\rar{\rho}\ar[out=180,in=180]{dd}[swap]{\psi^{A\otimes
          Q}}&\dgHom[A\otimes Q]{\P{A\otimes S_q}}{\P{A\otimes
          S_{q'}}}\dar{\pi_i}\ar[out=0,in=0]{dd}{\psi^{A\otimes Q}}\\
      A\otimes\Hom[Q]{\P{S_q}[i]}{\P{S_{q'}}[0]}\rar{\rho}\dar{1_A\otimes
        (p^0\circ?)}&\Hom[A\otimes Q]{\P{A\otimes S_q}[i]}{\P{A\otimes
          S_{q'}}[0]}\dar{(1_A\otimes p^0)\circ?}
      \\
      A\otimes\Hom[Q]{\P{S_q}[i]}{S_{q'}}\rar{\rho}&\Hom[A\otimes Q]{\P{A\otimes
          S_q}[i]}{A\otimes S_{q'}}
    \end{tikzcd}
  \]
  Indeed, for $i\in\ZZ$ and
  \[
    (a^j\otimes f^j)_{j\in\ZZ}\in\GdgHom[A\otimes Q]{\P{A\otimes
        S_q}}{\P{A\otimes
        S_{q'}}}[-i]=\prod_{j\in\ZZ}A\otimes\Hom[Q]{\P{S_q}[j]}{\P{S_{q'}}[j-i]},
  \]
  on the one hand
  \begin{align*}
    (a^j\otimes f^j)_{j\in\ZZ} &= (a^j\otimes(x \mapsto f_v^j(x))_{v\in Q})_{j\in\ZZ}\\
                               &\stackrel{\rho}{\mapsto} ((b\otimes x \mapsto a^jb\otimes f_v^j(x))_{v\in Q})_{j\in\ZZ}&&\eqref{eq:rho}\\
                               &\stackrel{\psi}{\mapsto}(b\otimes x \mapsto a^ib\otimes p^0(f_v^i(x)))_{v\in Q}),&&\eqref{eq:def-psi},\ \eqref{eq:proj-AQ}
  \end{align*}
  while on the other hand,
  \begin{align*}
    (a^j\otimes f^j)_{j\in\ZZ}  &\stackrel{\psi}{\mapsto}(1_A\otimes p^0)\circ(a^i\otimes f^i)&&\eqref{eq:transferred-psi}\\
                                &=a^i\otimes(p^0\circ f^i)\\
                                &\stackrel{\rho}{\mapsto}(b\otimes x\mapsto a^ib\otimes p^0(f_v^i(x)))_{v\in Q}.&&\eqref{eq:rho}
  \end{align*}
  
  Finally, we verify that the square~\eqref{eq:key-square} commutes. For $a\in
  A$ and a homogeneous morphism
  \[
    f\in\dgHom[Q]{\P{S_q}}{\P{S_{q'}}}[-i]=\prod_{j\in\ZZ}\Hom[Q]{\P{S_q}[j]}{\P{S_{q'}}[j-i]},
  \]
  of degree $-i$ we have
  \begin{align*}
    \psi^{A\otimes Q}(\varphi(a\otimes f))&=\psi^{A\otimes Q}((a\otimes f^j)_{j\in\ZZ})&&\eqref{eq:graded-map-formula}\\
                                          &=(1_A\otimes p^0)\circ(a\otimes f^i)&&\eqref{eq:transferred-psi}\\
                                          &=a\otimes(p^0\circ f^i)\\
                                          &=a\otimes\psi^Q(f)&&\eqref{eq:def-psi}\\
                                          &=(1_A\otimes\psi^{Q})(a\otimes f).
  \end{align*}
  This finishes the proof of the theorem.
\end{proof}

\begin{remark}
  \label{rmk:proof-of-main-thm}
  The proof of \Cref{thm:Q-shaped-bimodules} only makes use of the formal
  properties of the complete projective resolutions of the objects in the
  standard set of compact generators $S_Q(\kk)\subseteq\DerCat<Q>{\kk}$ (see
  \Cref{subsubsec:stalks}), which therefore can be replaced by a different such
  set as long as it consists of strictly perfect objects in the sense
  of~\cite[Def.~5.3]{HJ24} (see~\cite[Thm.~B]{HJ24} and compare also with the
  proof of~\cite[Thm.~A]{GHJS24} where the existence of a tilting object of
  $\DerCat<Q>{\kk}$ is leveraged).
\end{remark}

\begin{remark}
  Recall \Cref{setup:Q-shaped-dercats}. Let
  $i\colon\dgP{S_Q(\kk)}\cong\kk\otimes\dgP{S_Q(\kk)}\longrightarrow
  A\otimes\dgP{S_Q(\kk)}$ be the dg functor induced by the structure map $\kk\to
  A$. Under the identification provided by \Cref{thm:Q-shaped-bimodules}, we
  obtain the standard derived adjunction $(\mathbb{L}i_!\dashv
  i^*)$~\cite[Sec.~6]{Kel94}
  \[
    \begin{tikzcd}
      \mathbb{L}i_!\colon\DerCat<Q>{\kk}\simeq\DerCat{\dgP{S_Q(\kk)}}\rar[shift
      left]&\DerCat{A\otimes \dgP{S_Q}}\simeq\DerCat<Q>{A}\noloc i^*,\lar[shift
      left]
    \end{tikzcd}
  \]
  compare with~\cite[Notation~2.3.3 and~Prop.~3.1.4]{GHJS24} but beware of the
  discrepancy in notation. In particular, each exact triangle ${X\to Y\to Z\to
    X[1]}$ in $\DerCat<Q>{A}$ induces an exact triangle ${i^*X\to i^*Y\to
    i^*Z\to (i^*X)[1]}$ in $\DerCat<Q>{\kk}$ and hence there are long exact
  sequences of functors
  \[
    \begin{tikzcd}[column sep=tiny,row sep=small]
      \cdots\rar&\Hom[\DerCat<Q>{\kk}]{-}{i^*X}\rar\dar{\wr}&\Hom[\DerCat<Q>{\kk}]{-}{i^*Y}\rar\dar{\wr}&\Hom[\DerCat<Q>{\kk}]{-}{i^*Z}\rar\dar{\wr}&\cdots\\
      \cdots\rar&\Hom[\DerCat<Q>{A}]{\mathbb{L}i_!(-)}{X}\rar&\Hom[\DerCat<Q>{A}]{\mathbb{L}i_!(-)}{Y}\rar&\Hom[\DerCat<Q>{A}]{\mathbb{L}i_!(-)}{Z}\rar&\cdots
    \end{tikzcd}
  \]
  Similarly, each exact triangle ${L\to M\to N\to L[1]}$ in $\DerCat<Q>{\kk}$
  induces an exact triangle ${\mathbb{L}i_!L\to \mathbb{L}i_!M\to
    \mathbb{L}i_!N\to (\mathbb{L}i_!L)[1]}$ in $\DerCat<Q>{A}$ and hence there
  are long exact sequences of functors
  \[
    \begin{tikzcd}[column sep=tiny,row sep=small]
      \cdots\rar&\Hom[\DerCat<Q>{A}]{\mathbb{L}i_!N}{-}\rar\dar{\wr}&\Hom[\DerCat<Q>{A}]{\mathbb{L}i_!M}{-}\rar\dar{\wr}&\Hom[\DerCat<Q>{A}]{\mathbb{L}i_!L}{-}\rar\dar{\wr}&\cdots\\
      \cdots\rar&\Hom[\DerCat<Q>{\kk}]{N}{i^*(-)}\rar&\Hom[\DerCat<Q>{\kk}]{M}{i^*(-)}\rar&\Hom[\DerCat<Q>{\kk}]{L}{i^*(-)}\rar&\cdots;
    \end{tikzcd}
  \]
  compare with~\cite[Thms.~9 and~10]{IM15}, where
  \[
    \mathbb{H}_{M}^n(X)\coloneqq\Hom[\DerCat<Q>{A}]{\mathbb{L}i_!M}{X[n]}\cong\Hom[\DerCat<Q>{\kk}]{M}{i^*X[n]}
  \]
  is regarded as an `$M$-indexed cohomology of $X$,' as well as with the
  cohomology functors $\mathbb{H}_{[q]}^n=\mathbb{H}_{S_q}^n$ (for $n>0$)
  associated to the stalk objects
  $S_Q(\kk)\subseteq\DerCat<Q>{\kk}$~\cite[Def.~7.11]{HJ22}.
\end{remark}
 
The following example shows that in general it is unavoidable to consider higher
structures in \Cref{thm:Q-shaped-bimodules}.

\begin{example}
  Let $\kk$ be a field and $Q=\kk[\partial]/(\partial^N)$, $N\geq3$, a truncated
  polynomial algebra with non-quadratic relations. Then
  $\DerCat<Q>{\kk}=\sMod{Q}$ is the stable category of all $Q$-modules
  (see~\Cref{ex:stableQMod}) and it admits the unique simple $Q$-module $S$ as a
  compact generator. Since $Q$ is not a Koszul algebra, the (non-positive
  truncation of the) derived endomorphism algebra of $S$ cannot be
  quasi-isomorphic to its cohomology (viewed as a dg algebra with vanishing
  differential), see~\cite[Prop.~1]{Kel02a}.
\end{example}

\begin{remark}
  \label{rmk:Q-shaped-RHom}
  Let $\A$ be a small dg category. \Cref{thm:Q-shaped-bimodules} suggests to
  consider the dg category
  \[
    \RHom[c]{\dgDerCat{\A^\op}}{\dgDerCat<Q>{\kk}}\stackrel{\eqref{eq:derived-EW-theorem}}{\simeq}\dgDerCat{\A\Lotimes\dgP{S_Q(\kk)}}
  \]
  as a `$Q$-shaped derived dg category' of $\A$. Such a generalisation of the
  $Q$-shaped derived category from algebras (with many objects) to dg algebras
  (with many objects) does not seem to be immediately possible from the specific
  setup considered in~\cite{HJ22,HJ24}.
\end{remark}

\begin{remark}
  \label{rmk:Q-shaped-LuriesTP}
  \Cref{thm:Q-shaped-bimodules} also suggests the following approach to
  constructing $Q$-shaped derived categories within the framework of
  $\infty$-categories developed by Joyal~\cite{Joy02,Joyb,Joya,Joy},
  Lurie~\cite{Lur09,Lur17} and many others. Since this construction is not
  needed in the sequel, we only provide the reader with a rough sketch of our
  ideas but with enough references to the literature in order to make them
  rigorous. Let $\A$ be a small dg category. The derived $\infty$-category
  $\infDerCat{\A}$ of $\A$ is a $\kk$-linear compactly generated stable
  $\infty$-category\footnote{The case of $\kk$-linear derived
    $\infty$-categories of dg algebras is discussed in~\cite[Sec.~2.4]{Chr22};
    the apparent generalisation to small dg categories follows from the version
    of~\cite[Thm.~4.3.3.17]{Lur09} for `algebras with many objects' combined
    with~\cite[Prop.~4.3.2.5 and Cor.~4.3.2.8.]{Lur17} (see
    also~\cite[Rmk.~1.4.4.3, Prop.~4.8.3.22 and Rmk.~4.8.3.23]{Lur17}).} and
  hence it can be regarded as an object of the $\infty$-category of $\kk$-linear
  presentable stable $\infty$-categories\footnote{By `$\kk$-linear presentable
    stable $\infty$-category' we mean a presentable stable $\infty$-category
    with an action of the derived $\infty$-category $\infDerCat{\kk}$, viewed as
    a symmetric monoidal $\infty$-category with the derived tensor product over
    $\kk$, see~\cite[Thm.~7.1.2.13]{Lur17} and~\cite[Variant~D.1.5.1]{Lur18SAG}.
    Such $\infty$-categories can be regarded as enhanced variants of Neeman's
    well-generated ($\kk$-linear) triangulated categories,
    see~\cite[Prop.~A.3.7.6]{Lur09}, \cite[Prop.~1.3.4.22]{Lur17}
    and~\cite[Prop.~6.10]{Ros05}.} and $\kk$-linear colimit-preserving functors
  between them. The latter $\infty$-category is endowed with a
  closed\footnote{See for example~\cite[Sec.~4.1]{HSS17}.} symmetric monoidal
  structure
  \[
    (\C,\D)\longmapsto\LTP{\C}{\D}=\LTP{\C}[\kk]{\D}
  \]
that we refer to as \emph{Lurie's ($\kk$-linear) tensor
  product}~\cite[Sec.~D.2]{Lur18SAG}; for example, if $\A$ and $\B$ are small dg
categories, then\footnote{See~\cite[Sec.~4.3, Thms.~4.8.4.1, 4.3.2.7, 4.8.4.6,
  and~p.~738]{Lur17}.}
  \[
    \LTP{\infDerCat{\A}}{\infDerCat{\B}}\simeq\infDerCat{\A\Lotimes\B},
  \]
  where $(\A,\B)\mapsto\A\Lotimes B$ denotes the derived tensor product of dg
  categories (\Cref{subsec:dg_cats_up_to_quasi-eq} and compare also with~\Cref{rmk:Deligne-type-tensor-product}). In these terms,
  \Cref{thm:Q-shaped-bimodules} suggests to consider the $\kk$-linear compactly
  generated stable $\infty$-category
  \[
    \infDerCat<Q>{\A}\coloneqq\LTP{\infDerCat{\A}}{\infDerCat<Q>{\kk}}\simeq\infDerCat{\A\Lotimes\dgP{S_Q(\kk)}}
  \]
  as a `$Q$-shaped derived $\infty$-category of $\A$,' where
  $\infDerCat<Q>{\kk}\coloneqq\infDerCat{\dgP{S_Q(\kk)}}$ for the sake of
  concreteness. Finally, it is worth mentioning that the construction
  \[
    \C\longmapsto\infDerCat<Q>{\C}\coloneqq\LTP{\C}{\infDerCat<Q>{\kk}}
  \]
  is functorial with respect to $\kk$-linear colimit-preserving functors and
  enjoys the following properties:
  \begin{itemize}
  \item The tensor product $\LTP{\C}{\infDerCat<Q>{\kk}}$ is defined when $\C$
    is presentable but not necessarily compactly generated, for example when
    $\C=\infDerCat{\G}$ is the derived $\infty$-category of a $\kk$-linear
    Grothendieck category $\G$, see~\cite[Prop.~1.3.5.21]{Lur17}
    and~\cite[Ex~D.1.3.9]{Lur18SAG}.
  \item The construction $\C\mapsto\LTP{\C}{\infDerCat<Q>{\kk}}$ preserves
    limits\footnote{Since $\kk$-linear compactly generated stable
      $\infty$-categories are dualisable, see~\cite[Def.~4.6.1.1]{Lur17}
      and~\cite[Rmk.~D.7.7.6]{Lur18SAG} and then
      also~\cite[Lemma~4.6.1.6]{Lur17} and~\cite[Prop.~5.2.3.5]{Lur09}.} and
    colimits\footnote{Since Lurie's tensor product is closed, see
      also~\cite[Prop.~5.2.3.5]{Lur09}.} of $\kk$-linear presentable stable
    $\infty$-categories along  $\kk$-linear colimit-preserving functors.
  \item If $\C$ and $\infDerCat<Q>{\kk}$ are equipped with sufficiently `nice'
    $t$-structures,\footnote{In slightly more detail, we require that $\C$ is
      equivalent to the $\infty$-category of spectrum objects of a $\kk$-linear
      Grothendieck prestable $\infty$-category in the sense
      of~\cite[Def.~C.1.4.2]{Lur18SAG}, see also~\cite[Prop.~C.3.1.1]{Lur18SAG}.
      This assumption is satisfied, for example, by the standard $t$-structure
      on the derived category of a $\kk$-linear Grothendieck
      category~\cite[Ex.~C.1.4.5]{Lur18SAG}, and by $t$-structures generated by a
      compact silting object~\cite[Exs.~C.1.5.11 and~C.1.4.4]{Lur18SAG}.} then
    the Lurie tensor product $\LTP{\C}{\infDerCat<Q>{\kk}}$ inherits an equally
    `nice' $t$-structure~\cite[Rmks.~D.2.1.3, D.2.1.6 and~D.2.3.4, and
    Prop.~D.2.2.1]{Lur18SAG}.
  \end{itemize}
\end{remark}

\subsection{Consequences of \Cref{thm:Q-shaped-bimodules}}
\label{subsec:consequences}

We explain some consequences of \Cref{thm:Q-shaped-bimodules} that follow from
the properties of To{\"e}n's internal $\operatorname{Hom}$ functor recalled in
\Cref{subsec:Toen}.

\begin{setup}
  \label{setup:Q-shaped-dercats-flat}
  In addition to the assumptions in \Cref{setup:Q-shaped-dercats}, from now on
  we assume that the underlying $\kk$-module of the algebra $A$ is flat (this
  condition is always satisfied if $\kk$ is a field); in particular,
  \Cref{thm:Q-shaped-bimodules} applies to the $Q$-shaped derived category of
  $A$. Moreover, since the functor $A\otimes-$ preserves quasi-isomorphisms of
  dg $\kk$-modules, for every dg category $\B$ there is an
  isomorphism\footnote{Indeed, $A\Lotimes \B\coloneqq A\otimes \QQ\B$ where
    $\QQ\B\to\B$ is a cofibrant replacement of $\B$ in the Tabuada model
    structure (in particular a quasi-equivalence)~\cite[Sec.~4]{Toe07}. Since
    the functor $A\otimes-$ preserves quasi-equivalences, the induced dg functor
    $A\Lotimes\B=A\otimes \QQ\B\to A\otimes\B$ is a quasi-equivalence, hence an
    isomorphism in $\Hqe$.}
  \[
    A\Lotimes\B\cong A\otimes\B
  \]
  in $\Hqe$. Notice that if in addition the underlying $\kk$-module of $A$ is
  finitely generated, then it is projective since the ground commutative ring
  $\kk$ is assumed to be noetherian~\cite[Thm.~3.2.7]{Wei94}.
\end{setup}

Recall that, by Rickard's Theorem~\cite{Ric89,Ric91} (see
~\cite[Thm.~6.1]{Kel07} for a comprehensive statement), two $\kk$-algebras $A_1$
and $A_2$ whose underlying $\kk$-modules are flat are derived equivalent (in the
usual sense) if and only if there is an isomorphism in $\Hqe$
\[
  \dgDerCat{A_1}\cong\dgDerCat{A_2},
\]
see also~\cite[Thm.~9.2]{Kel94} and~\cite[Thm.~3.12]{Kel06} for more general
statements.

The following general result can be used to construct equivalences between
$Q$-shaped derived categories for different choices of $Q$,
see~\Cref{ex:repetitive}.

\begin{theorem}
  \label{thm:Q-shaped-derived-invariance}
  Let $Q_1$ and $Q_2$ be small $\kk$-categories and $A_1$ and $A_2$ two
  $\kk$-algebras, and assume that these satisfy the assumptions in
  \Cref{setup:Q-shaped-dercats-flat}. Suppose that there are isomorphisms in
  $\Hqe$
  \[
    \dgDerCat<Q_1>{\kk}\cong\dgDerCat<Q_2>{\kk}\qquad\text{and}\qquad\dgDerCat{A_1}\cong\dgDerCat{A_2}.
  \]
  Then, there is an isomorphism in $\Hqe$
  \[
    \dgDerCat<Q_1>{A_1}\cong\dgDerCat<Q_2>{A_2}
  \]
  and, consequently, an equivalence of triangulated categories
  \[
    \DerCat<Q_1>{A_1}\simeq\DerCat<Q_2>{A_2}.
  \]
\end{theorem}
\begin{proof}
  \Cref{thm:Q-shaped-bimodules} yields isomorphisms in $\Hqe$
  \[
    \dgDerCat<Q_1>{A_1}\stackrel{\ref{thm:Q-shaped-bimodules}}{\cong}\dgDerCat{A_1\otimes\dgP{S_{Q_1}(\kk)}}\stackrel{\eqref{eq:-derived-invariance-bimodules}}{\cong}\dgDerCat{A_2\otimes\dgP{S_{Q_2}(\kk)}}\stackrel{\ref{thm:Q-shaped-bimodules}}{\cong}\dgDerCat<Q_2>{A_2}.
  \]
  By passing to the corresponding $0$-th cohomology categories the required
  equivalence of triangulated categories follows.
\end{proof}

\begin{corollary}[Derived invariance]
  \label{coro:Q-shaped-derived-invariance-fixed-Q}
  Let $Q$ be a small $\kk$-category and $A_1$ and $A_2$ two $\kk$-algebras, and
  assume that these satisfy the assumptions in
  \Cref{setup:Q-shaped-dercats-flat}. Suppose that there is an isomorphism in
  $\Hqe$
  \[
    \dgDerCat{A_1}\cong\dgDerCat{A_2}.
  \]
  Then, there is an isomorphism in $\Hqe$
  \[
    \dgDerCat<Q>{A_1}\cong\dgDerCat<Q>{A_2}
  \]
  and, consequently, an equivalence of triangulated categories
  \[
    \DerCat<Q>{A_1}\simeq\DerCat<Q>{A_2}.
  \]
\end{corollary}
\begin{proof}
  Immediate from \Cref{thm:Q-shaped-derived-invariance} by taking $Q_1=Q=Q_2$.
\end{proof}

\begin{corollary}[Shapeshifting]
  \label{coro:Q-shaped-derived-invariance-fixed-A}
  Let $Q_1$ and $Q_2$ be small $\kk$-categories and $A$ a $\kk$-algebra, and
  assume that these satisfy the assumptions in
  \Cref{setup:Q-shaped-dercats-flat}. Suppose that there is an isomorphism in
  $\Hqe$
  \[
    \dgDerCat<Q_1>{\kk}\cong\dgDerCat<Q_2>{\kk}.
  \]
  Then, there is an isomorphism in $\Hqe$
  \[
    \dgDerCat<Q_1>{A}\cong\dgDerCat<Q_2>{A}
  \]
  and, consequently, an equivalence of triangulated categories
  \[
    \DerCat<Q_1>{A}\simeq\DerCat<Q_2>{A}.
  \]
\end{corollary}
\begin{proof}
  Immediate from \Cref{thm:Q-shaped-derived-invariance} by taking $A_1=A=A_2$.
\end{proof}

\begin{remark}
  \label{rmk:universal-derived-equivalences} In
  \Cref{coro:Q-shaped-derived-invariance-fixed-A}, the isomorphism
  $\dgDerCat<Q_1>{A}\cong\dgDerCat<Q_2>{A}$ in $\Hqe$ is `universal' since it is
  induced by an isomorphism $\dgDerCat<Q_1>{\kk}\cong\dgDerCat<Q_2>{\kk}$ in
  $\Hqe$, compare with \Cref{rmk:Deligne-type-tensor-product}.
\end{remark}

The following result leverages the existence of alternative compact
generators in the $Q$-shaped derived category of the ground commutative ring in
order to obtain alternative descriptions of general $Q$-shaped derived
categories.

\begin{theorem}[Change of compact generators]
  \label{thm:independence-of-generators}
  With the assumptions in \Cref{setup:Q-shaped-dercats-flat}, let
  $G\subseteq\DerCat<Q>{\kk}$ be any set of compact generators and let
  $\dgP{G}\subseteq\dgDerCat<Q>{\kk}$ be the full dg subcategory spanned by the
  objects in $G$. Then, there is an isomorphism in $\Hqe$
  \[
    \dgDerCat<Q>{A}\cong\dgDerCat{A\otimes\dgP{G}}
  \]
  and, consequently, an equivalence of triangulated categories
  \[
    \DerCat<Q>{A}\simeq\DerCat{A\otimes\dgP{G}}.
  \]
\end{theorem}
\begin{proof}
  Keller's Recognition Theorem (\Cref{subsec:Keller}) yields isomorphisms in
  $\Hqe$ (see~\eqref{eq:Kellers-recognition})
  \[
    \dgDerCat{\dgP{S_Q(\kk)}}\cong\dgDerCat<Q>{\kk}\cong\dgDerCat{\dgP{G}}.
  \]
  Consequently, in view of the functoriality of To{\"e}n's internal
  $\operatorname{Hom}$, there are isomorphisms in $\Hqe$
  \begin{align*}
    \dgDerCat<Q>{A}&\cong\dgDerCat{A\otimes\dgP{S_Q(\kk)}}&&\text{\Cref{thm:Q-shaped-bimodules}}\\
                   &\cong\dgDerCat{A\otimes\dgP{G}}.&&\eqref{eq:-derived-invariance-bimodules}
  \end{align*}
  The required equivalence of triangulated categories is obtained by passing to
  the corresponding $0$-th cohomology categories.
\end{proof}

Let $\T$ be a triangulated category with small coproducts. Recall
from~\cite[Def.~4.1]{AI12} (see also~\cite{HKM02}) that a full subcategory
$G\subseteq\T$ consisting of a set of compact generators is called a
\emph{silting subcategory} if
\[
  \forall X,Y\in G,\ \forall i>0,\qquad \T(X,\Suspension[i]{Y})=0.
\]
We say that $G\subseteq\T$ is a \emph{tilting subcategory} if moreover
\[
  \forall X,Y\in G,\ \forall i\neq0,\qquad \T(X,\Suspension[i]{Y})=0.
\]
In either case, the \emph{graded category $G^\bullet$ associated to $G$} has the
same set of objects as $G$ and the graded
$\kk$-modules of morphisms
\[
  G^\bullet(X,Y)\coloneqq\coprod_{i\in\ZZ}\T(X,\Suspension[i]{Y}),\qquad X,Y\in G,
\]
which are endowed with the apparent composition law.

\begin{corollary}[Silting generators]
  \label{thm:silting-generators}
  With the assumptions in \Cref{setup:Q-shaped-dercats-flat}, suppose that there
  exists a silting subcategory $G\subseteq\DerCat<Q>{\kk}$ and
  $\dgP{G}\subseteq\dgDerCat<Q>{\kk}$ the full dg subcategory spanned by the
  objects in $G$. Then, the $Q$-shaped category $\DerCat<Q>{A}$ admits a silting
  subcategory whose associated graded category is isomorphic to
  $A\otimes\H[\leq0]{\dgP{G}}$.
\end{corollary}
\begin{proof}
  \Cref{thm:independence-of-generators} yields an equivalence of (dg enhanced)
  triangulated categories
  \[
    \DerCat<Q>{A}\simeq\DerCat{A\otimes\dgP{G}}.
  \]
  The representable dg $(A\otimes\dgP{G})-$modules provide a set of compact
  generators of $\DerCat{A\otimes\dgP{G}}$ and, since $A$ is assumed to be flat
  as a $\kk$-module, there is an isomorphism of graded categories
  \[
    \H{A\otimes\dgP{G}}\cong A\otimes\H{\dgP{G}}=A\otimes\H[\leq0]{\dgP{G}}.
  \]
  Therefore the representable dg modules in $\DerCat{A\otimes\dgP{G}}$ form a
  silting subcategory with the required associated graded category.
\end{proof}

\begin{corollary}[Tilting generators]
  \label{thm:tilting-generators}
  With the assumptions in \Cref{setup:Q-shaped-dercats-flat}, let
  $G\subseteq\DerCat<Q>{\kk}$ be a tilting subcategory and
  $\dgP{G}\subseteq\dgDerCat<Q>{\kk}$ be the full dg subcategory spanned by the
  objects in $G$. Then, the $Q$-shaped category $\DerCat<Q>{A}$ admits a tilting
  subcategory that is isomorphic to $A\otimes G$ as a $\kk$-category. In
  particular, there is an equivalence of (dg enhanced) triangulated categories
  \[
    \DerCat<Q>{A}\simeq\DerCat{A\otimes G}.
  \]
\end{corollary}
\begin{proof}
  Immediate from~\Cref{thm:silting-generators}: Since
  $G\subseteq\DerCat<Q>{\kk}$ is a tilting subcategory,
  $\H{\dgP{G}}=\H[0]{\dgP{G}}$ and hence there are isomorphisms in
  $\Hqe$~\cite[Sec.~9]{Kel94}
  \[
    \dgP{G}\cong\H[0]{\dgP{G}}\cong G.\qedhere
  \]
\end{proof}

We conclude this section with some examples that illustrate how the previous
results can be used to recover known descriptions of $Q$-shaped derived
categories for specific choices of $Q$.

\begin{example}
  \label{ex:repetitive}
  Let $\kk$ be a field. Suppose given two basic and elementary
  finite-dimensional $\kk$-algebras $\Lambda_1$ and $\Lambda_2$ of finite global
  dimension and assume, moreover, that they are derived equivalent. Let $Q_1$
  and $Q_2$ be the repetitive $\kk$-categories in the sense
  of~\cite[Sec.~II.2]{Hap88} associated to $\Lambda_1$ and $\Lambda_2$,
  respectively. By Happel's Theorem~\cite[Thm.~II.4.9]{Hap88}, there are
  equivalences of triangulated categories
  \begin{align*}
    \DerCat<Q_1>{\kk} &= \sMod{Q_1}
                        \simeq \DerCat{\Lambda_1}
                        \simeq \DerCat{\Lambda_2}
                        \simeq \sMod{Q_2}
                        = \DerCat<Q_2>{\kk},
  \end{align*}
  see also~\Cref{ex:stableQMod}. These equivalences lift to the dg level since
  they are induced by the existence of tilting objects in each of the
  triangulated categories involved (see for example ~\cite[Thm.~8.6]{Kel07} that
  also explains why Happel's Theorem extends to unbounded derived categories).
  Consequently there is an isomorphism in $\Hqe$
  \[
    \dgDerCat<Q_1>{\kk}\cong \dgDerCat<Q_2>{\kk}
  \]
  and hence~\Cref{thm:Q-shaped-derived-invariance} applies in this case (see
  also \Cref{ex:GHJS}).
\end{example}

\begin{example}
  \label{ex:m-periodic}
  Let $\kk$ be a field and $Q$ the path $\kk$-category of the quiver $(m\geq 2)$
  
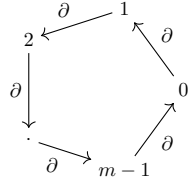
\begin{center}
    \begin{tikzpicture}[scale=0.75,every node/.style={transform shape}]
      \node (0) at (0*72:1.5) {$0$}; \node (1) at (1*72:1.5) {$1$}; \node (2) at
      (2*72:1.5) {$2$}; \node (3) at (3*72:1.5) {$\cdot$}; \node (4) at
      (4*72:1.5) {$m-1$}; \draw[->] (0)--node[auto,swap]{$\partial$}(1);
      \draw[->] (1)--node[auto,swap]{$\partial$}(2); \draw[->]
      (2)--node[auto,swap]{$\partial$}(3); \draw[->]
      (3)--node[auto,swap]{$\partial$}(4); \draw[->]
      (4)--node[auto,swap]{$\partial$}(0);
    \end{tikzpicture}
  \end{center}
  subject to the relation $\partial\circ\partial=0$. Thus, $\Mod{Q}$ is the
  category of strictly $m$-periodic complexes of $\kk$-vector spaces and
  $\DerCat<Q>{\kk}=\stable{\Mod{Q}}$ is the derived category of all such,
  see~\cite[Sec.~1.5]{HJ24b} and also~\Cref{ex:stableQMod}. It is well-known and
  easy to see that there are isomorphisms in $\DerCat<Q>{\kk}$
  \[
    \Sigma^q(S_0)\cong S_q,\qquad q=0,1,\dots,m-1.
  \]
  Therefore the simple $Q$-module $S=S_0$ is a compact generator of
  $\DerCat<Q>{\kk}$ with graded endomorphism algebra
  \[
    \coprod_{i\in\ZZ}\Hom[\DerCat<Q>{\kk}]{S}{\Sigma^i(S)}\cong\kk\langle
    u^{\pm1}\rangle,\qquad |u|=m,
  \]
  the algebra of Laurent polynomials in a single variable of degree $m$, see for
  example~\cite[Sec.~5.5]{JKM22} (notice that there is an equivalence of
  categories with suspension
  \[
    \DerCat<Q>{\kk}=\sMod{Q}\simeq\prod_{p=0}^{m-1}\Mod{\kk},
  \]
  where the suspension on the right-hand side is given by cyclic permutation of
  the factors). The latter graded algebra is known to be intrinsically formal
  (see~\cite[Cor.~4.8]{Sai23} for a more general statement), which is to say
  that any dg algebra with this cohomology is formal (=quasi-isomorphic to its
  cohomology as a dg algebra). \Cref{thm:independence-of-generators} then yields
  an equivalence of (dg enhanced) triangulated categories\footnote{We regard the
    graded algebra $A\otimes\kk\langle u^{\pm1}\rangle$ as a dg algebra with
    vanishing differential and consider dg modules over it rather than cochain
    complexes of graded modules.}
  \[
    \DerCat<Q>{A}\simeq\DerCat{A\otimes\kk\langle u^{\pm1}\rangle},
  \]
  which is a well-known description of the derived category of $m$-periodic
  complexes of $A$-modules (see for example~\cite[Sec.~3.3]{Sai23}). In fact, in
  this case there is an equivalence of categories
  \[
    \dgMod*{A\otimes\kk\langle
      u^{\pm1}\rangle}\stackrel{\sim}{\longrightarrow}\Mod{A\otimes Q},
  \]
  where the left-hand side is the category of dg $(A\otimes\kk\langle
  u^{\pm1}\rangle)$-modules, see for example~\cite[Prop.~3.36]{Sai23}.
\end{example}

\begin{example}
  \label{ex:IKM}
  Under the additional assumptions in \Cref{setup:Q-shaped-dercats-flat}, the
  description of the derived category of $N$-complexes of $A$-modules $(N\geq2)$
  of Iyama--Kato--Miyachi~\cite[Thm.~F]{IKM17} can be deduced from
  \Cref{thm:tilting-generators} thanks to the existence of a tilting object in
  the derived category of $N$-complexes of $\kk$-modules~\cite[Prop.~4.7]{IKM17}
  --- which is also a main ingredient in the proof of~\cite[Thm.~F]{IKM17} (see
  also~\cite[Sec.~1.5]{HJ24b}).
\end{example}

\begin{example}
  \label{ex:GHJS}
  \Cref{thm:tilting-generators} yields an alternative\footnote{But, after all,
    equivalent, see~\Cref{rmk:proof-of-main-thm}.} proof
  of~\cite[Thm.~A]{GHJS24}, which also relies on the existence of a tilting
  object in the $Q$-shaped derived category $\DerCat<Q>{\kk}$ for the specific
  family of small $\kk$-categories $Q$ considered therein~\cite[Thm.~1.3]{Yam13}
  (in these articles the ground ring $\kk$ is a field). Indeed, if
  $T\in\DerCat<Q>{\kk}$ is Yamaura's tilting object,
  \Cref{thm:tilting-generators} yields an equivalence of triangulated categories
  \[
    \DerCat<Q>{A}\simeq\DerCat{A\otimes\Gamma},\qquad\Gamma\coloneqq\End[\DerCat<Q>{\kk}]{T},
  \]
  which is precisely the statement of~\cite[Thm.~A]{GHJS24}. We note that the
  hypotheses in~\cite[Thm.~1.3]{Yam13}, and hence also those
  in~\cite[Thm.~A]{GHJS24}, allow for $Q$ to be the repetitive $\kk$-category
  (equivalently, the trivial extension with its canonical non-negative
  grading~\cite[II.2.4, Lemma]{Hap88}) of a basic and elementary
  finite-dimensional algebra of finite global dimension~\cite[Ex.~3.15]{Yam13}
  (see also~\Cref{ex:repetitive}).
\end{example}


\subsection*{Acknowledgements}

The author wishes to thank, H.~Holm, P.~J{\o}rgensen and D.~Nkansah for useful
discussions on $Q$-shaped derived categories, H.~Holm and P.~J{\o}rgensen for bringing to his attention results in~\cite[Sec.~5]{HJ24} that are crucial to the
proof of \Cref{thm:Q-shaped-bimodules}, and P.~J{\o}rgensen for his comments on
earlier versions of this article. The author thanks Merlin Christ for the
  reference to ~\cite[Sec.~4.1]{HSS17}. Finally, the author thanks the anonymous
referee(s) for their useful comments on a previous version of this article that
lead, in particular, to the inclusion of \Cref{rmk:alternative_means}.

\printbibliography

\end{document}
